\documentclass[a4paper,11pt,reqno]{amsart}

\usepackage[english]{babel}
\usepackage{amsmath}
\usepackage{amssymb}
\usepackage{amsfonts}
\usepackage{amsthm}
\usepackage{float}
\usepackage{stmaryrd}
\usepackage{graphics}
\usepackage{graphicx}
\usepackage{subfig}
\usepackage{datetime}
\usepackage[colorlinks=false]{hyperref}
\usepackage{color}
\newcommand{\eps}{\varepsilon}
\newcommand{\epsn}{{\varepsilon_n}}
\renewcommand{\d}{\,\mathrm{d}}

\newcommand{\mc}{\mathcal}
\newcommand{\diff}{\!\setminus\!}
\newcommand{\elleps}{{\ell^\eps}}
\newcommand{\ellepsd}{{\dot{\ell}^\eps}}
\newcommand{\ueps}{{u^\eps}}
\newcommand{\uepst}{{u_t^\eps}}
\newcommand{\uepsx}{{u_x^\eps}}
\newcommand{\weps}{{w^\eps}}
\newcommand{\wepsd}{{\dot{w}^\eps}}
\newcommand{\reps}{{r^\eps}}
\newcommand{\ellepsn}{{\ell^\epsn}}
\newcommand{\ellepsnd}{{\dot{\ell}^\epsn}}
\newcommand{\uepsn}{{u^\epsn}}
\newcommand{\uepsnt}{{u_t^\epsn}}
\newcommand{\uepsnx}{{u_x^\epsn}}
\newcommand{\wepsn}{{w^\epsn}}
\newcommand{\wepsnd}{{\dot{w}^\epsn}}
\newcommand{\repsn}{{r^\epsn}}
\def\enne{\mathbb{N}}
\def\zeta{\mathbb{Y}}

\def\erre{\mathbb{R}}

\renewcommand{\to}{\rightarrow}

\numberwithin{equation}{section}
\newtheorem{thm}{Theorem}[section]
\newtheorem{defi}[thm]{Definition}
\newtheorem{prop}[thm]{Proposition}
\newtheorem{lemma}[thm]{Lemma}
\newtheorem{cor}[thm]{Corollary}

\theoremstyle{definition}
\begingroup
\newtheorem{rmk}[thm]{Remark}

\endgroup

\theoremstyle{remark}
\begingroup
\endgroup

\begin{document}
	
	\author{Filippo Riva}
	
	\title[Quasistatic limit via vanishing inertia and viscosity]{On the approximation of quasistatic evolutions for the debonding of a thin film via vanishing inertia and viscosity}
	
	\begin{abstract}
		In this paper we contribute to studying the issue of quasistatic limit in the context of Griffith's theory by investigating a one-dimensional debonding model. It describes the evolution of a thin film partially glued to a rigid substrate and subjected to a vertical loading. Taking friction into account and under suitable assumptions on the toughness of the glue, we prove that, in contrast to what happens in the undamped case, dynamic solutions converge to the quasistatic one when inertia and viscosity go to zero, except for a possible discontinuity at the initial time. We then characterise the size of the jump by means of an asymptotic analysis of the debonding front.
	\end{abstract}

	\maketitle
		
	{\small
		\par
	}

	\pagenumbering{arabic}
	
\medskip

\tableofcontents

	\section*{Introduction}	
	In most of the models within the theory of linearly elastic fracture mechanics the evolution process is assumed to be quasistatic, namely the body is at equilibrium at every time. This postulate seems to be reasonable since inertial effects can be	neglected if the speed of external loading is very slow with respect to the one of internal oscillations. However, its mathematical proof is really far from being achieved in the general framework, due to the high complexity and diversity of the phenomena under consideration. We can rephrase the problem, commonly referred as quasistatic limit issue, as follows: is it true that quasistatic evolutions can be approximated by dynamic ones when the external loading becomes slower and slower, or equivalently the speed of internal vibrations becomes faster and faster? Nowadays only partial results on the theme are available; we refer for instance to \cite{LazRos} and \cite{Roub} for damage models, to \cite{DMSca} in a case of perfect plasticity and to \cite{LazNarkappa}, \cite{LazNar} for the undamped version of the debonding model we analyse in this work. The issue of quasistatic limit has also been studied in a finite-dimensional setting where, starting from the works \cite{Ago}, \cite{Zan} and with the contribution of \cite{Nar15}, an almost complete understanding on the topic has been reached in \cite{ScilSol}. A common feature appearing both in finite both in infinite dimension is the validation of the quasistatic approximation only in presence of a damping term in the dynamic model. Because of this consideration, in this paper we resume a particular kind of debonding model previously inspected in \cite{LazNar} taking in addition viscosity into account. In \cite{LazNar} the authors proved that, due to lack of viscosity, the resulting limit evolution turns out not to be quasistatic, even in the case of a constant toughness of the glue between the film and the substrate. Thanks to friction, we are instead able to give a positive answer to the quasistatic limit question in the model under examination, covering the case of quite general toughnesses.\par 
	We refer to \cite{BurrKell}, \cite{Fre90}, \cite{Hela}, \cite{Helb}, \cite{Helbook} for an introduction to one-dimensional debonding models from an engineering point of view; a first analysis on the quasistatic limit in these kind of models is instead developed in \cite{DouMarCha08}, \cite{LBDM12}. The rigorous mathematical formulation we will follow throughout the paper has been introduced in \cite{DMLazNar16}, used in \cite{LazNarkappa}, \cite{LazNar}, \cite{LazNarinitiation} for the undamped case, and adopted in \cite{Riv} and \cite{RivNar} for well-posedness results in the damped case.\par 
	The mechanical system we consider describes the debonding of a perfectly flexible and inextensible thin film initially glued to a flat rigid substrate and subjected to a vertical loading $w$ at an endpoint. The deformation of the film takes place in the half plane $\{(x,y)\mid x\ge 0\}$ and at time $t\ge0$ is given by $(x,0)\mapsto(x+h(t,x),u(t,x))$, where the functions $h$ and $u$ are the horizontal and the vertical displacement of the point $(x,0)$, respectively. In the reference configuration the debonded region is $\{(x,0)\mid 0\le x< \ell(t)\}$, where $\ell$ is a nondecreasing function representing the debonding front and satisfying $\ell(0)=\ell_0>0$. See Figure \ref{fig1}.
	\begin{figure}
		\includegraphics[scale=0.7]{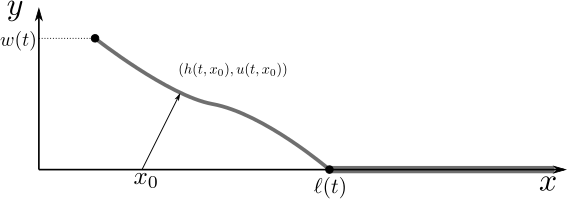}
		\caption{The deformation of the film at time $t$ is represented by the displacement $(x_0,0)\mapsto(x_0+h(t,x_0), u(t,x_0))$. The function $w(t)$ is the vertical loading, while $\ell(t)$ is the debonding front.}\label{fig1}
	\end{figure}
	By linear approximation and inextensibility of the film the horizontal displacement $h$ is uniquely determined by the vertical one $u$, so the only unknowns of the problem are $u$ and the debonding front $\ell$.\par 
	Since our aim is the analysis of the behaviour of the system in the case of slow loading and slow initial velocity we introduce a small parameter $\eps$ in the model, so that the vertical displacement $u_\eps$ (we add the subscript to stress the dependence on $\eps$) solves the dynamic problem:
		\begin{equation}
	\label{problem0}
	\begin{cases}
	 (u_\eps)_{tt}(t,x)-(u_\eps)_{xx}(t,x)+\nu (u_\eps)_t(t,x)=0, \quad& t > 0 \,,\, 0<x<\ell_\eps(t),  \\
	u_\eps(t,0)=w(\eps t), &t>0, \\
	u_\eps(t,\ell_\eps(t))=0,& t>0,\\
	u_\eps(0,x)=u_0(x),\quad&0<x<\ell_0,\\
	(u_\eps)_t(0,x)=\eps u_1(x),&0<x<\ell_0,
	\end{cases}
	\end{equation}
	where $u_0$ and $u_1$ are given initial data, while $\nu\ge 0$ is a parameter which tunes the friction of air. The evolution of the debonding front $\ell_\eps$ will be established later on by suitable energy criteria, in the context of Griffith's theory.\par 
	In the above equation the speed of the waves is one, while the one of the external loading and initial velocity is of order $\eps$. Actually we are interested in studying the limit as the speed of internal vibrations becomes faster and faster, so we need to consider the time-rescaled functions $\Big(\ueps(t,x),\elleps(t)\Big):=\Big(u_\eps\left(\frac t\eps,x\right),\ell_\eps\left(\frac t\eps\right)\Big)$, which solve:
	\begin{equation}
	\label{problem1}
	\begin{cases}
	\eps^2 u^\eps_{tt}(t,x)-u^\eps_{xx}(t,x)+\nu \eps\uepst(t,x)=0, \quad& t > 0 \,,\, 0<x<\elleps(t),  \\
	\ueps(t,0)=w(t), &t>0, \\
	\ueps(t,\elleps(t))=0,& t>0,\\
	\ueps(0,x)=u_0(x),\quad&0<x<\ell_0,\\
	\uepst(0,x)=u_1(x),&0<x<\ell_0,
	\end{cases}
	\end{equation}
	plus Griffith's criterion \eqref{Griffith} ruling the growth of $\elleps$. In this rescaled setting internal waves move with speed $1/\eps$, while the speed of the loading $w$ and of the velocity $u_1$ is of order one. The aim of the work is thus to analyse the limit as $\eps$ goes to $0^+$ of this rescaled pair $(\ueps,\elleps)$.\par 
	The paper is organised as follows. Section \ref{sec1} deals with the dynamic model: we first introduce the energy criteria governing the evolution of the debonding front and we present the concept of (dynamic) Griffith's criterion. Then we collect the known results, proved in \cite{DMLazNar16}, \cite{LazNar}, \cite{Riv} and \cite{RivNar}, on the time-rescaled problem \eqref{problem1} coupled with Griffith's criterion. In particular Theorem~\ref{exuniq} states that there exists a unique dynamic evolution $(\ueps,\elleps)$ for the debonding model.\par 
	In Section \ref{sec2} we instead analyse the notion of quasistatic evolution in our framework; we refer to \cite{MieRou15} for the general topic of quasistatic and rate-independent processes. We then provide an existence and uniqueness result under suitable assumptions, see Theorem~\ref{exuniqquas}.\par 
	The last two Sections are devoted to the study of the limit of the pair $(\ueps,\elleps)$ as $\eps$ goes to $0^+$. In Section \ref{sec3} we exploit the presence of the viscous term in the wave equation to gain uniform bounds and estimates for the vertical displacement $\ueps$ and the debonding front $\elleps$. Finally in Section \ref{sec4} we prove that if $\nu>0$, namely when friction is taken into account, and requiring suitable assumptions on the toughness of the glue, the limit of dynamic evolutions $(\ueps,\elleps)$ exists and it is the quasistatic evolution we previously found in Section \ref{sec2}, except for a possible discontinuity at time $t=0$ produced by an excess of initial potential energy. We conclude the paper by giving a characterisation of this initial jump.
	
	\section*{Notations}
	In this preliminary Section we collect some notations we will use several times throughout the paper. Similar notations have been  introduced in \cite{DMLazNar16}, \cite{RivNar} and also used in \cite{LazNarkappa}, \cite{LazNar}, \cite{Riv}.
	\begin{rmk}
		In the paper every function in the Sobolev space $W^{1,p}(a,b)$, for $-\infty<a<b<+\infty$ and $p\in[1,+\infty]$, is always identified with its continuous representative on $[a,b]$.\par 
		Furthermore the derivative of any function of real variable is denoted by a dot (i.e. $\dot{f}$, $\dot{\ell}$, $\dot{\varphi}$, $\dot{u}_0$), regardless of whether it is a time or a spatial derivative.
	\end{rmk}
	Fix $\ell_0>0$, $\eps>0$ and consider a function $\elleps \colon [0,+\infty)\to [\ell_0,+\infty)$, which will play the role of the (rescaled) debonding front, satisfying:
	\begin{subequations}\label{elle}
		\begin{equation}\label{ellea}
		\elleps\in C^{0,1}([0,+\infty)) ,
		\end{equation}
		\begin{equation}\label{elleb}
		\elleps(0)=\ell_0\mbox{ and } 0\le\ellepsd(t)< 1/\eps \mbox{ for a.e. }t\in (0,+\infty).
		\end{equation}
	\end{subequations}	
	Given such a function, for $t\in[0,+\infty)$ we introduce:
	\begin{equation}\label{phipsidef}
	\varphi^\eps(t):= t {-} \eps\elleps(t) \,\mbox{, }\quad \psi^\eps(t):=t{+}\eps\elleps(t),
	\end{equation} 
	and we define:
	\begin{equation*}
	\omega^\eps\colon [\eps\ell_0,+\infty) \to [-\eps\ell_0,+\infty) , \quad
	\omega^\eps(t):=\varphi^\eps\circ{\psi^\eps}^{-1}(t).
	\end{equation*}
		We notice that $\psi^\eps$ is a bilipschitz function since by \eqref{elleb} it holds $1\le \dot{\psi}^\eps(t)< 2$ for almost every time, while $\varphi^\eps$ turns out to be Lipschitz with $0<\dot\varphi^\eps(t)\le 1$ almost everywhere. Hence $\varphi^\eps$ is invertible and the inverse is absolutely continuous on every compact interval contained in $\varphi^\eps([0,+\infty))$. As a byproduct we get that $\omega^\eps$ is Lipschitz too and for a.e. $t\in(\eps\ell_0,+\infty)$ we have:
		\begin{equation*}
		0<\dot\omega^\eps(t)=\frac{1-\eps\ellepsd({\psi^\eps}^{-1}(t))}{1+\eps\ellepsd({\psi^\eps}^{-1}(t))}\le 1.
		\end{equation*}
		So also $\omega^\eps$ is invertible and the inverse is absolutely continuous on every compact interval contained in $\omega^\eps([0,+\infty))$. Moreover, given $j\in\enne\cup\{0\}$, and denoting by $(\omega^\eps)^j$ the composition of $\omega^\eps$ with itself $j$ times (whether it is well defined) one has:
		\begin{equation}\label{dercomposition}
			\frac{\d}{\d t}(\omega^\eps)^{j+1}(\psi^\eps(t))=\frac{1{-}\eps\ellepsd(t)}{1{+}\eps\ellepsd(t)}\frac{\d}{\d t}(\omega^\eps)^{j}(\varphi^\eps(t)),\text{ for a.e. }t\in({\varphi^\eps}^{-1}((\omega^\eps)^{-j}({-}\eps\ell_0)),+\infty).
		\end{equation}
		It will be useful to define the sets:
		\begin{align*}
		&\Omega^\eps := \{ (t,x)\mid t>0\,,\, 0 < x < \elleps(t)\},\\
		&\Omega^\eps_T :=\{ (t,x)\in\Omega^\eps\mid t<T\}.
		\end{align*}
	For $(t,x)\in\Omega^\eps$ we also introduce: 
	\begin{equation}\label{rettangoli}
		\begin{aligned}
		&R^\eps_+(t,x)=\bigcup\limits_{j=0}^{m}R^\eps_{2j}(t,x),\\
		&R^\eps_-(t,x)=\bigcup\limits_{j=0}^{ n}R^\eps_{2j+1}(t,x),
		\end{aligned}
	\end{equation}
	\begin{figure}
		\includegraphics[scale=0.65]{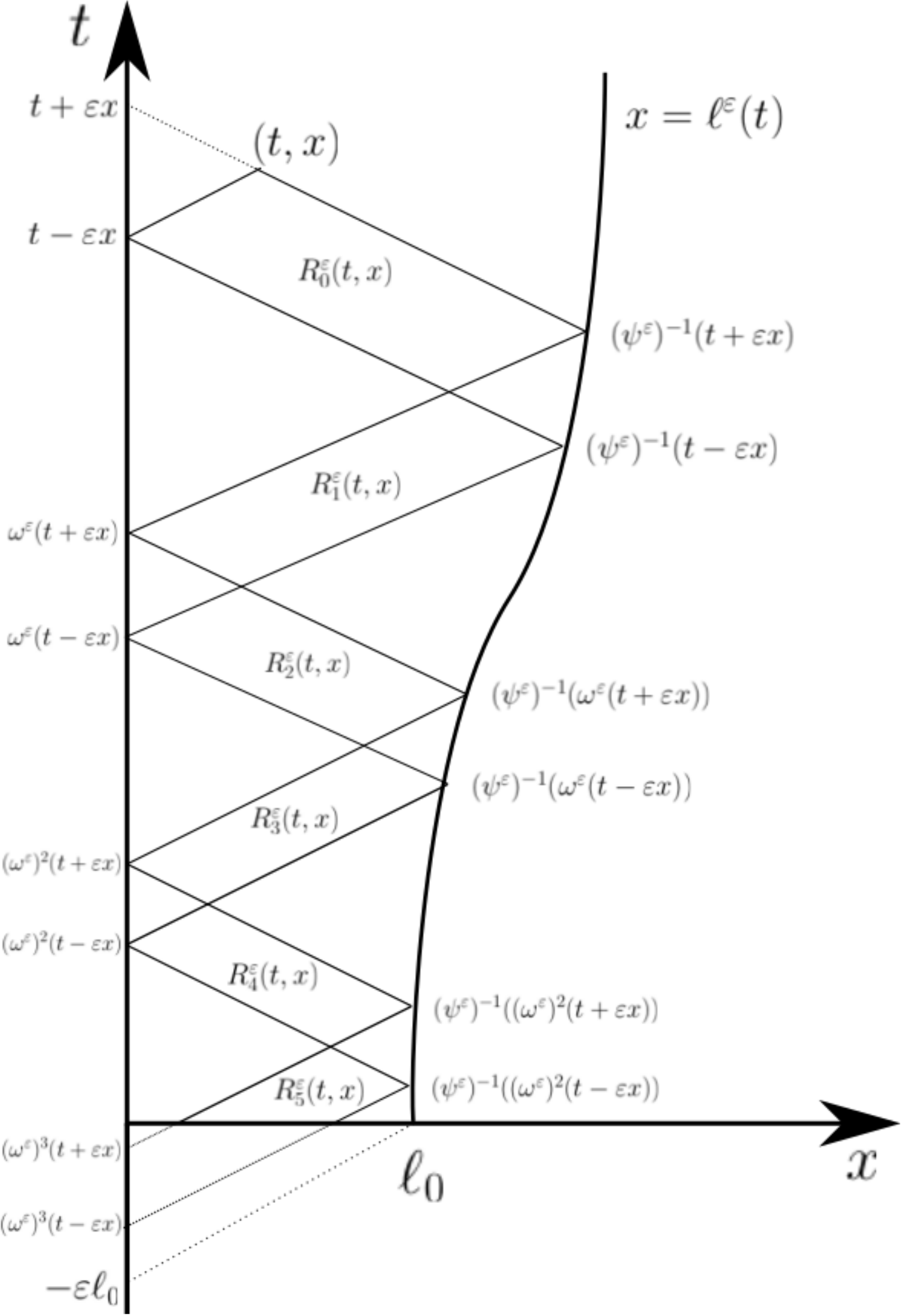}
		\caption{The sets $R^\eps_i(t,x)$ in the particular situation $\eps=1/2$, and with a choice of $(t,x)$ for which $m=2$, $n=2$.}\label{fig2}
	\end{figure}
	In order to avoid the cumbersome definitions of $m=m(\eps,t,x)$, $n=n(\eps,t,x)$ and $R^\eps_i(t,x)$ we refer to the very intuitive Figure \ref{fig2}.\\
	Finally, for $k\in\enne$, let us define the spaces:
	\begin{gather*}
	\widetilde{L}{^2}(\Omega^\eps) := \{u \in L^2_{\textnormal{loc}}(\Omega^\eps)\mid u \in L^2(\Omega^\eps_T) \mbox{ for every } T>0\},\\
	\widetilde{H}{^k}(\Omega^\eps) := \{u \in H^k_{\textnormal{loc}}(\Omega^\eps)\mid u \in H^k(\Omega^\eps_T) \mbox{ for every } T>0\},\\
	\widetilde H^k(0,+\infty) :=\{u \in H^k_{\textnormal{loc}}(0,+\infty) \mid u\in H^k(0,T) \, \text{ for every } T>0\},\\
	\widetilde C^{0,1} ([\ell_0,+\infty)) := \{ u \in C^0([\ell_0,+\infty)) \mid u \in C^{0,1}([\ell_0,X]) \text{ for every } X>\ell_0 \} .
	\end{gather*}
	We say that a family $\mc F$ is bounded in $\widetilde H^k(0,+\infty)$ if for every $T>0$ there exists a positive constant $C_T$ such that $\Vert u\Vert_{H^k(0,T)}\le C_T$ for every $u\in\mc F$. We say that a sequence $\{u_n\}_{n\in\enne}$ converges strongly (weakly) to $u$ in $\widetilde{H}^k(0,+\infty)$ if for every $T>0$ one has $u_n\to u$ ($u_n\rightharpoonup u$) in ${H}^k(0,T)$ as $n\to+\infty$.
	
	\section{Time-rescaled dynamic evolutions}\label{sec1}
	In this Section we give a presentation on the notion of dynamic evolutions for the considered debonding model, gathering all the known results about its well-posedness, see Theorems~\ref{exuniq}, \ref{contdependence} and Remark~\ref{regularity}. We refer to \cite{DMLazNar16}, \cite{LazNar}, \cite{Riv} and \cite{RivNar} for more details.\par 
	We fix $\nu\ge 0$, $\ell_0>0$ and we actually consider a slight generalisation of the rescaled problem \eqref{problem1}:
		\begin{equation}
	\label{problem2}
	\begin{cases}
	\eps^2 u^\eps_{tt}(t,x)-u^\eps_{xx}(t,x)+\nu \eps\uepst(t,x)=0, \quad& t > 0 \,,\, 0<x<\elleps(t),  \\
	\ueps(t,0)=\weps(t), &t>0, \\
	\ueps(t,\elleps(t))=0,& t>0,\\
	\ueps(0,x)=u^\eps_0(x),\quad&0<x<\ell_0,\\
	\uepst(0,x)=u^\eps_1(x),&0<x<\ell_0,
	\end{cases}
	\end{equation}
	in which also loading term and initial data depend on the (small) parameter $\eps>0$. We require they satisfy the following regularity assumptions:
	 \begin{subequations}\label{bdryregularity}
	 	\begin{equation}	
	 	\weps \in \widetilde  H^1(0,+\infty),	
	 	\end{equation}
	 	\begin{equation}
	 	u^\eps_0 \in H^1(0,\ell_0) , \quad u^\eps_1 \in L^2(0,\ell_0) .
	 	\end{equation}
	 \end{subequations}	
	and they fulfill the compatibility conditions:
	\begin{equation}\label{compatibility0}
	u^\eps_0(0)=\weps(0), \quad u^\eps_0(\ell_0)=0.
	\end{equation}
	\begin{rmk}
		By solution of problem \eqref{problem2} we mean an $\widetilde{H}^1(\Omega^\eps)$ function which solves the (damped) wave equation in the sense of distributions in $\Omega^\eps$ and attains the boundary values $\weps$ and $u^\eps_0$ in the sense of traces, while the initial velocity $u^\eps_1$ in the sense of $H^{-1}(0,\ell_0)$. 
	\end{rmk}
	To state the rules governing the evolution of the rescaled debonding front $\elleps$ we consider the following rescaled energies, defined for $t\in[0,+\infty)$:
	\begin{subequations}
		\begin{equation}\label{kinpot}
			\mc E^\eps(t)=\frac 12\int_{0}^{\elleps(t)}\left(\eps^2\uepst(t,\sigma)^2+\uepsx(t,\sigma)^2\right)\d \sigma;
		\end{equation}
		\begin{equation}\label{frict}
			\mc A^\eps(t)=\nu\int_{0}^{t}\int_{0}^{\elleps(\tau)}\eps\uepst(\tau,\sigma)^2\d\sigma\d\tau;
		\end{equation}
		\begin{equation}
			\mc W^\eps(t)=\int_{0}^{t}\wepsd(\tau)\uepsx(\tau,0)\d\tau.
		\end{equation}	
	\end{subequations}
	They represent the sum of kinetic and potential energy, the energy dissipated by friction and the work of the external loading, respectively. We postulate that our model is governed by an energy-dissipation balance and a maximum dissipation principle, namely the pair $(\ueps,\elleps)$ has to satisfy:
	\begin{subequations}\label{energy}
		\begin{equation}\label{eb}
			\mc E^\eps(t)+\mc A^\eps(t)+\int_{\ell_0}^{\elleps(t)}\kappa(\sigma)\d \sigma+\mc W^\eps(t)=\mc E^\eps(0),\quad\text{ for every }t\in[0,+\infty),
		\end{equation}
		where $\kappa\colon[\ell_0,+\infty)\to(0,+\infty)$ is a measurable function representing the toughness of the glue, and:
		\begin{equation}
			\ellepsd(t)=\max\left\{\alpha\in[0,1/\eps)\mid \kappa(\elleps(t))\alpha=G^\eps_{\eps\alpha}(t)\alpha \right\},\quad\text{ for a.e. }t\in(0,+\infty),
		\end{equation}
	\end{subequations}
	where $G^\eps_{\eps\alpha}$ is the (rescaled) dynamic energy release rate at speed $\eps\alpha\in[0,1)$. Formally it can be seen as the opposite of the derivative of the total energy $\mc E^\eps+\mc A^\eps+\mc W^\eps$ with respect to the elongation of $\elleps$ and it measures the amount of energy spent by the debonding process. We refer to \cite{DMLazNar16}, \cite{Fre90}, \cite{LazNar}, \cite{RivNar} and \cite{Slep} for its rigorous definition and properties.\\	
	As proved in \cite{RivNar} the two principles \eqref{energy} are equivalent to dynamic Griffith's criterion:
	\begin{equation}\label{Griffith}
		\begin{cases}
		0\le\ellepsd(t)<1/\eps,\\
		G^\eps_{\eps\ellepsd(t)}(t)\le\kappa(\elleps(t)),\\
		\left[G^\eps_{\eps\ellepsd(t)}(t)-\kappa(\elleps(t))\right]\ellepsd(t)=0,
		\end{cases}\quad\quad\text{ for a.e. }t\in(0,+\infty).
	\end{equation}
	The first row is an irreversibility condition, which ensures that the debonding front can only increase; the second one is a stability condition, and says that the dynamic energy release rate cannot exceed the threshold given by the toughness; the third one is simply the energy-dissipation balance \eqref{eb}.
	\begin{rmk}
		We recall that Griffith's criterion \eqref{Griffith} is also equivalent to an ordinary differential equation for $\elleps$:
		\begin{equation}\label{explicitder}
			\ellepsd(t)=\frac 1\eps \max\left\{\frac{G^\eps_0(t)-\kappa(\elleps(t))}{G^\eps_0(t)+\kappa(\elleps(t))} ,0\right\},\quad\text{ for a.e. }t\in(0,+\infty).
		\end{equation}
		See \cite{DMLazNar16} and \cite{RivNar} for more details. We will not make use of this formula in this work, but it can be helpful for further analysis and researches on the topic.
	\end{rmk}
	Before presenting the known results about the coupled problem \eqref{problem2}\&\eqref{Griffith} we introduce a function which will be useful in a representation formula for the solution of \eqref{problem2}. Given a function $\Theta\in \widetilde{L}^2(\Omega^\eps)$ we define:
    \begin{equation}\label{H}
    	H^\eps[\Theta](t,x):=\frac 12\left[\iint_{R^\eps_+(t,x)}\Theta(\tau,\sigma)\d\sigma\d\tau-\iint_{R^\eps_-(t,x)}\Theta(\tau,\sigma)\d\sigma\d\tau\right], \text{ for }(t,x)\in\Omega^\eps,
    \end{equation} 
    where $R^\eps_\pm(t,x)$ are as in \eqref{rettangoli}. Here are listed the main properties of $H^\eps$, under the assumption that $\elleps$ satisfies \eqref{elle}:
    \begin{prop}\label{Hprop}
    	Let $\Theta\in \widetilde{L}^2(\Omega^\eps)$, then the function $H^\eps[\Theta]$ introduced in \eqref{H} is continuous on $\overline{\Omega^\eps}$ and belongs to $\widetilde{H}^1(\Omega^\eps)$. Moreover, setting $H^\eps[\Theta]\equiv 0$ outside $\overline{\Omega^\eps}$, it belongs to $C^0([0,+\infty);H^1(0,+\infty))$ and to $ C^1([0,+\infty);L^2(0,+\infty))$.\\
    	Furthermore for a.e. $t\in(0,+\infty)$ one has:
    	\begin{equation}\label{Hxzero}
    	\begin{aligned}
    	H^\eps[\Theta]_x(t,0)=&\quad\sum\limits_{j=0}^{m^\eps-1}\frac{\d}{\d t}(\omega^\eps)^j(t)\int_{{\psi^\eps}^{-1}((\omega^\eps)^j(t))}^{(\omega^\eps)^j(t)}\Theta\left(\tau,\frac{(\omega^\eps)^j(t)-\tau}{\eps}\right)\d\tau\\
    	&-\sum\limits_{j=0}^{m^\eps-1}\frac{\d}{\d t}(\omega^\eps)^{j+1}(t)\int_{(\omega^\eps)^{j+1}(t)}^{{\psi^\eps}^{-1}((\omega^\eps)^j(t))}\Theta\left(\tau,\frac{\tau-(\omega^\eps)^{j+1}(t)}{\eps}\right)\d\tau\\
    	&+I^\eps_1(t),
    	\end{aligned}
    	\end{equation}
    	where $m^\eps\!=\!m^\eps(t)$ is the only natural number (including $0$) such that $(\omega^\eps)^{m^\eps}(t)\in[0,(\omega^\eps)^{{-}1}(0))$, while $I^\eps_1$ is defined as follows:
    	\begin{equation*}
    		I^\eps_1(t)=
    		\frac{\d}{\d t}(\omega^\eps)^{m^\eps}(t)\int_{0}^{(\omega^\eps)^{m^\eps}(t)}\Theta\left(\tau,\frac{(\omega^\eps)^{m^\eps}(t)-\tau}{\eps}\right)\d\tau,\quad\quad\text{if }(\omega^\eps)^{m^\eps}(t)\in[0,\eps\ell_0),
    	\end{equation*}
    	while if $(\omega^\eps)^{m^\eps}(t)\in[\eps\ell_0,(\omega^\eps)^{{-}1}(0))$ it is defined in this other way: 
    	\begin{equation*}
    		I^\eps_1(t)=\frac{\d}{\d t}(\omega^\eps)^{m^\eps}(t)\!\!\!\!\!\!\!\!\!\!\!\!\!\!\int\limits_{{\psi^\eps}^{-1}((\omega^\eps)^{m^\eps}(t))}^{(\omega^\eps)^{m^\eps}(t)}\!\!\!\!\!\!\!\!\!\!\!\!\!\!\!\Theta\left(\tau,\frac{(\omega^\eps)^{m^\eps}(t)-\tau}{\eps}\right)\d\tau-\frac{\d}{\d t}(\omega^\eps)^{m^\eps+1}(t)\!\!\!\!\!\!\!\!\!\!\!\!\!\int\limits_{0}^{{\psi^\eps}^{-1}((\omega^\eps)^{m^\eps}(t))}\!\!\!\!\!\!\!\!\!\!\!\!\!\!\!\!\Theta\left(\tau,\frac{\tau-(\omega^\eps)^{m^\eps+1}(t)}{\eps}\right)\d\tau.
    	\end{equation*}
    	Finally for a.e. $t\in(0,+\infty)$ it holds:
    	\begin{equation}\label{Hxellet}
    		H^\eps[\Theta]_x(t,\elleps(t))=\frac{2}{1+\eps\ellepsd(t)}g^\eps[\Theta](t-\eps\elleps(t)),
    	\end{equation}
    	where for a.e. $s\in\varphi^\eps([0,+\infty))$ we define:
    	\begin{equation}\label{g}
    		\begin{aligned}
    		g^\eps[\Theta](s)=&\quad\frac 12\sum\limits_{j=0}^{n^\eps-1}\frac{\d}{\d s}(\omega^\eps)^j(s)\int_{{\psi^\eps}^{-1}((\omega^\eps)^j(s))}^{(\omega^\eps)^j(s)}\Theta\left(\tau,\frac{(\omega^\eps)^j(s)-\tau}{\eps}\right)\d\tau\\
    		&-\frac 12\sum\limits_{j=0}^{n^\eps-1}\frac{\d}{\d s}(\omega^\eps)^{j}(s)\int_{(\omega^\eps)^{j}(s)}^{{\psi^\eps}^{-1}((\omega^\eps)^{j-1}(s))}\Theta\left(\tau,\frac{\tau-(\omega^\eps)^{j}(s)}{\eps}\right)\d\tau\\
    		&+\frac 12\frac{\d}{\d s}(\omega^\eps)^{n^\eps}(s)I^\eps_2(s),
    		\end{aligned}
    	\end{equation}
    	where $n^\eps=n^\eps(s)$ is the only natural number (including $0$) such that $(\omega^\eps)^{n^\eps}(s)\in[-\eps\ell_0,\eps\ell_0)$, while
    	\begin{equation*}
    	I^\eps_2(s)=\begin{cases}\displaystyle
    	-\int_{0}^{{\psi^\eps}^{-1}((\omega^\eps)^{n^\eps-1}(s))}\Theta\left(\tau,\frac{\tau-(\omega^\eps)^{n^\eps}(s)}{\eps}\right)\d\tau,&\!\!\!\!\!\!\!\!\!\!\!\!\!\!\!\!\!\!\!\!\!\!\!\!\!\!\!\!\!\!\!\!\!\text{if }(\omega^\eps)^{n^\eps}(s)\in[-\eps\ell_0,0),\\
    	\displaystyle\int_{0}^{(\omega^\eps)^{n^\eps}(s)}\!\!\!\!\!\!\Theta\left(\tau,\frac{(\omega^\eps)^{n^\eps}(s)-\tau}{\eps}\right)\d\tau-\!\!\!\!\!\!\!\!\!\int\limits_{(\omega^\eps)^{n^\eps}(s)}^{{\psi^\eps}^{-1}((\omega^\eps)^{n^\eps-1}(s))}\!\!\!\!\!\!\!\!\!\!\!\!\!\!\!\!\Theta\left(\tau,\frac{\tau-(\omega^\eps)^{n^\eps}(s)}{\eps}\right)\d\tau,&\text{ otherwise}.
    	\end{cases}
    	\end{equation*}
    \end{prop}
\begin{proof}
	The regularity of $H^\eps[\Theta]$ can be proved in the same way of Lemma~1.11 in \cite{RivNar}, so we refer to it for the details. The validity of \eqref{Hxzero} is a straightforward matter of computations, see Figure \ref{fig2} for an intuition and also Remark~1.12 in \cite{RivNar}. To get \eqref{Hxellet}, always referring to Figure \ref{fig2} and to \cite{RivNar}, we compute:
	\begingroup\allowdisplaybreaks
	\begin{align*}
		&\quad H^\eps[\Theta]_x(t,\elleps(t))\\
		=&\quad\frac 12\sum\limits_{j=0}^{n^\eps-1}\left(\frac{\d}{\d t}(\omega^\eps)^j(t{-}\eps\elleps(t))+\frac{\d}{\d t}(\omega^\eps)^{j+1}(t{+}\eps\elleps(t))\right)\!\!\!\!\!\!\!\!\int\limits_{{\psi^\eps}^{-1}((\omega^\eps)^j(t-\eps\elleps(t)))}^{(\omega^\eps)^j(t-\eps\elleps(t))}\!\!\!\!\!\!\!\!\!\!\!\!\!\!\!\!\!\!\!\!\!\Theta\left(\tau,\frac{(\omega^\eps)^j(t{-}\eps\elleps(t)){-}\tau}{\eps}\right)\d\tau\\
		&-\frac 12\sum\limits_{j=0}^{n^\eps-1}\left(\frac{\d}{\d t}(\omega^\eps)^j(t{-}\eps\elleps(t))+\frac{\d}{\d t}(\omega^\eps)^{j+1}(t{+}\eps\elleps(t))\right)\!\!\!\!\!\!\!\!\!\!\!\int\limits_{(\omega^\eps)^{j}(t{-}\eps\elleps(t))}^{{\psi^\eps}^{-1}((\omega^\eps)^{j-1}(t{-}\eps\elleps(t)))}\!\!\!\!\!\!\!\!\!\!\!\!\!\!\!\!\!\!\!\!\!\Theta\left(\tau,\frac{\tau-(\omega^\eps)^{j}(t{-}\eps\elleps(t))}{\eps}\right)\d\tau\\
		&+\frac 12\left(\frac{\d}{\d t}(\omega^\eps)^{n^\eps}(t{-}\eps\elleps(t))+\frac{\d}{\d t}(\omega^\eps)^{n^\eps+1}(t{+}\eps\elleps(t))\right)I^\eps_2(t{-}\eps\elleps(t)),
	\end{align*}
	\endgroup
	and we conclude by using \eqref{dercomposition}.
\end{proof}
	\begin{lemma}
		Let $\Theta\in \widetilde{L}^2(\Omega^\eps)$ and consider $H^\eps[\Theta]$ and $g^\eps[\Theta]$ given by \eqref{H} and \eqref{g}, respectively. Then for a.e. $s\in\varphi^\eps([0,+\infty))\cap(0,+\infty)$ it holds:
		\begin{equation}\label{magic}
			g^\eps[\Theta](s)-\frac 12H^\eps[\Theta]_x(s,0)=-\frac 12\int_{s}^{{\varphi^\eps}^{-1}(s)}\Theta\left(\tau,\frac{\tau-s}{\eps}\right)\d\tau.
		\end{equation}
	\end{lemma}
\begin{proof}
	We start computing by means of \eqref{Hxzero} and \eqref{g}:
		\begingroup\allowdisplaybreaks
	\begin{align*}
		&\quad 2 g^\eps[\Theta](s)-H^\eps[\Theta]_x(s,0)\\
		=&\quad\sum\limits_{j=0}^{n^\eps-1}\frac{\d}{\d s}(\omega^\eps)^j(s)\int_{{\psi^\eps}^{-1}((\omega^\eps)^j(s))}^{(\omega^\eps)^j(s)}\Theta\left(\tau,\frac{(\omega^\eps)^j(s)-\tau}{\eps}\right)\d\tau\\
		&-\sum\limits_{j=0}^{n^\eps-1}\frac{\d}{\d s}(\omega^\eps)^{j}(s)\int_{(\omega^\eps)^{j}(s)}^{{\psi^\eps}^{-1}((\omega^\eps)^{j-1}(s))}\Theta\left(\tau,\frac{\tau-(\omega^\eps)^{j}(s)}{\eps}\right)\d\tau\\
		&-\sum\limits_{j=0}^{m^\eps-1}\frac{\d}{\d s}(\omega^\eps)^j(s)\int_{{\psi^\eps}^{-1}((\omega^\eps)^j(s))}^{(\omega^\eps)^j(s)}\Theta\left(\tau,\frac{(\omega^\eps)^j(s)-\tau}{\eps}\right)\d\tau\\
		&+\sum\limits_{j=0}^{m^\eps-1}\frac{\d}{\d s}(\omega^\eps)^{j+1}(s)\int_{(\omega^\eps)^{j+1}(s)}^{{\psi^\eps}^{-1}((\omega^\eps)^j(s))}\Theta\left(\tau,\frac{\tau-(\omega^\eps)^{j+1}(s)}{\eps}\right)\d\tau\\
		&+\frac{\d}{\d s}(\omega^\eps)^{n^\eps}(s)I^\eps_2(s)-I_1^\eps(s)=(\star).
	\end{align*}
	\endgroup
	There are only two cases to consider: $n^\eps(s)=m^\eps(s)$ or $n^\eps(s)=m^\eps(s)+1$. We prove the Lemma for the first case, being the other one analogous. So we have:
		\begingroup\allowdisplaybreaks
	\begin{equation}\label{telescopic}
		\begin{aligned}
		(\star)=&\quad\sum\limits_{j=0}^{n^\eps-1}\frac{\d}{\d s}(\omega^\eps)^{j+1}(s)\int_{(\omega^\eps)^{j+1}(s)}^{{\psi^\eps}^{-1}((\omega^\eps)^j(s))}\Theta\left(\tau,\frac{\tau-(\omega^\eps)^{j+1}(s)}{\eps}\right)\d\tau\\
		&-\sum\limits_{j=0}^{n^\eps-1}\frac{\d}{\d s}(\omega^\eps)^{j}(s)\int_{(\omega^\eps)^{j}(s)}^{{\psi^\eps}^{-1}((\omega^\eps)^{j-1}(s))}\Theta\left(\tau,\frac{\tau-(\omega^\eps)^{j}(s)}{\eps}\right)\d\tau\\
		&+\frac{\d}{\d s}(\omega^\eps)^{n^\eps}(s)I^\eps_2(s)-I_1^\eps(s)=(\star\star).
		\end{aligned}
	\end{equation}
	\endgroup
	Exploiting the fact that in \eqref{telescopic} there is now a telescopic sum and by using the explicit formulas of $I_1^\eps$ and $I_2^\eps$ given by Proposition~\ref{Hprop} we hence deduce:
	\begingroup\allowdisplaybreaks
	\begin{align*}
		(\star\star)=&\quad\frac{\d}{\d s}(\omega^\eps)^{n^\eps}(s)\int_{(\omega^\eps)^{n^\eps}(s)}^{{\psi^\eps}^{-1}((\omega^\eps)^{n^\eps-1}(s))}\Theta\left(\tau,\frac{\tau-(\omega^\eps)^{n^\eps}(s)}{\eps}\right)\d\tau-\int_{s}^{{\varphi^\eps}^{-1}(s)}\Theta\left(\tau,\frac{\tau-s}{\eps}\right)\d\tau\\
		&+\frac{\d}{\d s}(\omega^\eps)^{n^\eps}(s)\int_{0}^{(\omega^\eps)^{n^\eps}(s)}\Theta\left(\tau,\frac{(\omega^\eps)^{n^\eps}(s)-\tau}{\eps}\right)\d\tau\\
		&-\frac{\d}{\d s}(\omega^\eps)^{n^\eps}(s)\int_{(\omega^\eps)^{n^\eps}(s)}^{{\psi^\eps}^{-1}((\omega^\eps)^{n^\eps-1}(s))}\Theta\left(\tau,\frac{\tau-(\omega^\eps)^{n^\eps}(s)}{\eps}\right)\d\tau\\
		&-\frac{\d}{\d s}(\omega^\eps)^{n^\eps}(s)\int_{0}^{(\omega^\eps)^{n^\eps}(s)}\Theta\left(\tau,\frac{(\omega^\eps)^{n^\eps}(s)-\tau}{\eps}\right)\d\tau\\
		=&-\int_{s}^{{\varphi^\eps}^{-1}(s)}\Theta\left(\tau,\frac{\tau-s}{\eps}\right)\d\tau,
	\end{align*}
	\endgroup
	and we conclude.
\end{proof}
Finally we are in a position to state the main results about dynamic evolutions of the debonding model, namely solutions to coupled problem \eqref{problem2}\&\eqref{Griffith}. These two Theorems are obtained by collecting what the authors proved in \cite{DMLazNar16}, \cite{LazNar}, \cite{Riv} and \cite{RivNar}.
	\begin{thm}[\textbf{Existence and Uniqueness}]\label{exuniq}
		Fix $\nu\ge 0$, $\ell_0>0$, $\eps>0$, assume the functions $\weps$, $u_0^\eps$ and $u_1^\eps$ satisfy \eqref{bdryregularity}, \eqref{compatibility0} and let the toughness $\kappa$ be positive and satisfy the following property:
		\begin{equation*}
			\text{for every }x\in[\ell_0,+\infty)\text{ there exists }\delta=\delta(x)>0\text{ such that }\kappa\in C^{0,1}([x,x+\delta]).
		\end{equation*}
		Then there exists a unique pair $(\ueps,\elleps)$, with:
		\begin{itemize}
			\item $\elleps\in C^{0,1}([0,+\infty))$, $\elleps(0)=\ell_0$ and $0\le\ellepsd(t)<1/\eps$ for a.e. $t\in(0,+\infty)$,
			\item $\ueps\in \widetilde{H}^1(\Omega^\eps)$ and $\ueps(t,x)=0$ for every $(t,x)$ such that $x>\elleps(t)$,
		\end{itemize} 
		solution of the coupled problem \eqref{problem2}\&\eqref{Griffith}.\par 
		Moreover $\ueps$ has a continuous representative which fulfills the following representation formula:
		\begin{equation*}
			\ueps(t,x)=\begin{cases}
			\displaystyle\weps(t+\eps x)-\frac 1\eps f^\eps(t+\eps x)+\frac 1\eps f^\eps(t-\eps x)-\nu H^\eps[\uepst](t,x),& \text{if }(t,x)\in\overline{\Omega^\eps},\\
			0,&\text{otherwise},
			\end{cases}
		\end{equation*}
		where $f^\eps\in\widetilde{H}^1(-\eps\ell_0,+\infty)$ is defined by two rules:
		\begin{itemize}
			\item[(i)] $\displaystyle f^\eps(s)=\begin{cases}
			\displaystyle\eps\weps(s)-\frac \eps 2 u_0^\eps\left(\frac s\eps\right)-\frac{\eps^2}{2} \int_{0}^{s/\eps}u_1^\eps(\sigma)\d\sigma-\eps\weps(0)+\frac\eps 2 u_0^\eps(0),&\text{if }s\in(0,\eps\ell_0],\\
			\displaystyle\frac \eps 2 u_0^\eps\left(-\frac s\eps\right)-\frac{\eps^2}{2} \int_{0}^{-s/\eps}u_1^\eps(\sigma)\d\sigma-\frac\eps 2 u_0^\eps(0),&\text{if }s\in(-\eps\ell_0,0],
			\end{cases}$
			\item[(ii)] $\displaystyle\weps(s+\eps\elleps(s))-\frac 1\eps f^\eps(s+\eps\elleps(s))+\frac 1\eps f^\eps(s-\eps\elleps(s))=0,\quad\quad\quad$ for every $s\in(0,+\infty)$,
		\end{itemize}
		while $H^\eps$ is as in \eqref{H}.\par \noindent
		In particular it holds:
		\begin{equation*}
		\ueps\in C^0([0,+\infty);H^1(0,+\infty))\cap C^1([0,+\infty);L^2(0,+\infty)).
		\end{equation*}
		Furthermore one has:
		\begin{subequations}
			\begin{equation}\label{expux}
			\uepsx(t,0)=\eps\wepsd(t)-2\dot{f}^\eps(t)-\nu H^\eps[\uepst]_x(t,0),\quad\text{ for a.e. }t\in(0,+\infty),
			\end{equation}
			\begin{equation}
			\uepsx(t,\elleps(t))=-\frac{2}{1+\eps\ellepsd(t)}\Big[\dot{f}^\eps(t-\eps\elleps(t))+\nu g^\eps[\uepst](t-\eps\elleps(t))\Big],\quad\text{ for a.e. }t\in(0,+\infty),
			\end{equation}
		\end{subequations}
		and for $\alpha\in[0,1/\eps)$ the dynamic energy release rate can be expressed as:
		\begin{equation}\label{expgrif}
		G^\eps_{\eps\alpha}(t)=2\frac{1-\eps\alpha}{1+\eps\alpha}\Big[\dot{f}^\eps(t-\eps\elleps(t))+\nu g^\eps[\uepst](t-\eps\elleps(t))\Big]^2,\quad\text{ for a.e. }t\in(0,+\infty),
		\end{equation}
		where $g^\eps$ has been introduced in \eqref{g}.
	\end{thm}
	\begin{rmk}[\textbf{Regularity}]\label{regularity}
		If the data are more regular, namely:
		\begin{equation*}
				\weps \in \widetilde  H^2(0,+\infty),\quad u^\eps_0 \in H^2(0,\ell_0) , \quad u^\eps_1 \in H^1(0,\ell_0) ,
		\end{equation*}
		if the (positive) toughness $\kappa$ belongs to $\widetilde{C}^{0,1}([\ell_0,+\infty))$ and if besides \eqref{compatibility0} also the following first order compatibility conditions are satisfied:
		\begin{equation*}
		\begin{gathered}
		u^\eps_1(0)=\wepsd(0),\\
		\!\!\Big(u^\eps_1(\ell_0)\!=\!0,\,\, \dot{u}^\eps_0(\ell_0)^2\!\le \!2\kappa(\ell_0)\Big)\text{ or }\Big(u^\eps_1(\ell_0)\!\neq\! 0,\,\,\dot{u}^\eps_0(\ell_0)^2{-}\eps^2u^\eps_1(\ell_0)^2\!=\!2\kappa(\ell_0),\,\,\frac{\dot{u}^\eps_0(\ell_0)}{u^\eps_1(\ell_0)}\!<\!{-}\varepsilon\Big),
		\end{gathered}		
		\end{equation*}
		 then the solution $\ueps$ is in $\widetilde H^2(\Omega^\eps)$.
	\end{rmk}
	\begin{thm}[\textbf{Continuous Dependence}]\label{contdependence}
		Fix $\nu\ge 0$, $\ell_0>0$, $\eps>0$, assume the functions $\weps$, $u_0^\eps$ and $u_1^\eps$ satisfy \eqref{bdryregularity}, \eqref{compatibility0} and let the toughness $\kappa$ be positive and belong to $\widetilde{C}^{0,1}([\ell_0,+\infty))$. Consider sequences of functions $\{w^\eps_n\}_{n\in\enne}$, $\{{u_0^\eps}_n\}_{n\in\enne}$ and $\{{u_1^\eps}_n\}_{n\in\enne}$ satisfying \eqref{bdryregularity} and \eqref{compatibility0}, and let $(u^\eps_n,\ell^\eps_n)$ and $(\ueps,\elleps)$ be the solutions of coupled problem \eqref{problem2}\&\eqref{Griffith} given by Theorem~\ref{exuniq} corresponding to the data with and without the subscript $n$, respectively. If the following convergences hold true as $n\to+\infty$:
		\begin{equation*}
			{u_0^\eps}_n\to u_0^\eps\text{ in }H^1(0,\ell_0),\,\,{u_1^\eps}_n\to u_1^\eps\text{ in }L^2(0,\ell_0)\text{ and }w^\eps_n\to\weps \text{ in }\widetilde{H}^1(0,+\infty),
		\end{equation*}
		then for every $T>0$ one has as $n\to+\infty$:
		\begin{itemize}
			\item[-] $\ell^\eps_n\to\ell^\eps$ in $W^{1,1}(0,T)$;
			\item[-] $u^\eps_n\to \ueps$ uniformly in $[0,T]\times[0,+\infty)$;
			\item[-] $u^\eps_n\to \ueps$ in $H^1((0,T)\times(0,+\infty))$;
			\item[-] $u^\eps_n\to \ueps$ in $C^0([0,T];H^1(0,+\infty))$ and in $ C^1([0,T];L^2(0,+\infty))$;
			\item[-] $(u^\eps_n)_x(\cdot,0)\to \uepsx(\cdot,0)$ in $L^2(0,T)$.
		\end{itemize}
		\end{thm}
	
	\section{Quasistatic evolutions}\label{sec2}
	This Section is devoted to the analysis of quasistatic evolutions for the debonding model we are studying. We first introduce and compare two different notions of this kind of evolutions (we refer to \cite{BouFraMar08} or \cite{MieRou15} for a wide and complete presentation on the topic), then we prove an existence and uniqueness result under suitable assumptions, see Theorem~\ref{exuniqquas}.\par 
	Fix $\ell_0>0$; throughout this Section we consider a loading term $w\in C^0([0,+\infty))$ such that $w\in AC([0,T])$ for every $T>0$ and a toughness $\kappa\in C^0([\ell_0,+\infty))$ such that $\kappa(x)>0$ for every $x\ge\ell_0$.
	\begin{defi}\label{enev}
		Let $\lambda\colon[0,+\infty)\to[\ell_0,+\infty)$ be a nondecreasing function such that $\lambda(0)=\ell_0$ and let $v\colon[0,+\infty)\times[0,+\infty)\to\erre$ be a function which for every $t\in[0,+\infty)$ satisfies $v(t,\cdot)\in H^1(0,+\infty)$, $v(t,0)=w(t)$, $v(t,x)=0$ for $x\ge\lambda(t)$ and such that $v_x(t,0)$ exists for a.e. $t\in(0,+\infty)$. We say that such a pair $(v,\lambda)$ is an \textbf{energetic evolution} if for every $t\in[0,+\infty)$ it holds:
		\begin{itemize}
			\item[(S)]$\displaystyle \frac 12\int_{0}^{\lambda(t)}v_x(t,\sigma)^2\d \sigma+\int_{\ell_0}^{\lambda(t)}\kappa(\sigma)\d \sigma\le\frac 12 \int_{0}^{\hat{\lambda}}\dot{\hat{v}}(\sigma)^2\d \sigma+\int_{\ell_0}^{\hat{\lambda}}\kappa(\sigma)\d \sigma,$\\
			for every $\hat\lambda\ge\lambda(t)$ and for every $\hat{v}\in H^1(0,\hat\lambda)$ satisfying $\hat{v}(0)=w(t)$ and $\hat{v}(\hat\lambda)=0;$
			\item[(EB)] $\displaystyle \frac 12\int_{0}^{\lambda(t)}v_x(t,\sigma)^2\d \sigma+\int_{\ell_0}^{\lambda(t)}\kappa(\sigma)\d \sigma+\int_{0}^{t}\dot{w}(\tau)v_x(\tau,0)\d\tau=\frac 12\int_{0}^{\ell_0}v_x(0,\sigma)^2\d \sigma.$
		\end{itemize}
	\end{defi}\noindent
	Here (S) stands for (global) stability, while (EB) for energy(-dissipation) balance. Roughly speaking an energetic evolution is a pair which fulfills an energy-dissipation balance being at every time a global minimiser of the functional $(v,\lambda)\mapsto\frac 12 \int_{0}^{{\lambda}}{\dot v}(\sigma)^2\d \sigma+\int_{\ell_0}^{{\lambda}}\kappa(\sigma)\d \sigma$, which is sum of potential energy and energy dissipated to debond the film.\par 
	On the contrary, this second Definition deals with local minima of the total energy:
	\begin{defi}\label{quasev}
		Given $\lambda$ and $v$ as in Definition~\ref{enev}, we say that the pair $(v,\lambda)$ is a \textbf{quasistatic evolution} if:
		\begin{itemize}
			\item[(i)] $\lambda$ is absolutely continuous on $[0,T]$ for every $T>0$ and $\lambda(0)=\ell_0$;
			\item[(ii)] $\displaystyle v(t,x)=w(t)\left(1-\frac{x}{\lambda(t)} \right)\chi_{[0,\lambda(t)]}(x),$ for every $(t,x)\in[0,+\infty)\times[0,+\infty)$;
			\item[(iii)] the quasistatic version of Griffith's criterion holds true, namely:
			\begin{equation*}
				\begin{cases}
				\dot\lambda(t)\ge 0,\\
				\frac 12 \frac{w(t)^2}{\lambda(t)^2}\le \kappa(\lambda(t)),\\
				\left[\frac 12 \frac{w(t)^2}{\lambda(t)^2}-\kappa(\lambda(t))\right]\dot\lambda(t)=0,
				\end{cases}\quad\quad\text{ for a.e. }t\in(0,+\infty).
			\end{equation*}			
		\end{itemize}
	\end{defi}
Similarities with dynamic Griffith's criterion \eqref{Griffith} are evident, with the exception of the term $\frac 12 \frac{w(t)^2}{\lambda(t)^2}$ which requires some explanations: like in the dynamic case we can introduce the notion of quasistatic energy release rate as $G_{\textnormal{qs}}(t)=-\partial_\lambda \mc E_{\textnormal{qs}}(t)$, where the quasistatic energy $\mc E_{\textnormal{qs}}$ is simply the potential one, kinetic energy being negligible in a quasistatic setting. By means of (ii) we can compute $\mc E_{\textnormal{qs}}(t)=\frac 12\int_{0}^{\lambda(t)}v_x(t,\sigma)^2\d \sigma=\frac 12 \frac{w(t)^2}{\lambda(t)}$, from which we recover $G_{\textnormal{qs}}(t)=\frac 12 \frac{w(t)^2}{\lambda(t)^2}$. Thus (iii) is the correct formulation of quasistatic Griffith's criterion. \par 
For a reason which will be clear during the proof of next Proposition we introduce for $x\ge\ell_0$  the function $\phi_\kappa(x):=x^2\kappa(x)$. When needed we will assume one or more of the following hypothesis:
\begin{itemize}
	\item[(K1)] $\phi_\kappa$ is nondecreasing on $[\ell_0,+\infty)$;
	\item[(K2)] $\phi_\kappa$ is strictly increasing on $[\ell_0,+\infty)$;
	\item[(K3)] $\phi_\kappa$ is strictly increasing on $[\ell_0,+\infty)$ and $\dot{\phi}_\kappa(x)>0$ for a.e. $x\in(\ell_0,+\infty)$;
	\item[(KW)] $\displaystyle\lim\limits_{x\to+\infty}\phi_\kappa(x)>\frac 12\max\limits_{t\in[0,T]}w(t)^2$ for every $T>0$, and $\displaystyle\phi_\kappa(\ell_0)\ge \frac 12 w(0)^2$.
\end{itemize}
It is worth noticing that (K1) ensures local minima of the energy are actually global, as stated in Proposition~\ref{equiv}. Conditions (K2) and (K3) instead imply uniqueness of the minimum, see Proposition~\ref{explicitlambda}. Finally the first assumption in (KW) is related to the existence of such a minimum, replacing the role of coercivity of the energy, which can be missing.
\begin{prop}\label{equiv}
		Assume (K1). Then a pair $(v,\lambda)$ is an energetic evolution if and only if:
		\begin{itemize}
			\item[(o)] $\lambda$ is non decreasing on $[0,+\infty)$ and $\lambda(0)=\ell_0$;
			\item[(s1)] $\displaystyle v(t,x)=w(t)\left(1-\frac{x}{\lambda(t)} \right)\chi_{[0,\lambda(t)]}(x),$ for every $(t,x)\in[0,+\infty)\times[0,+\infty)$;
			\item[(s2)] $\displaystyle\frac 12 \frac{w(t)^2}{\lambda(t)^2}\le \kappa(\lambda(t)),$ for every $t\in[0,+\infty),$
			\item[(eb)] $\displaystyle\frac 12 \frac{w(t)^2}{\lambda(t)}+\int_{\ell_0}^{\lambda(t)}\kappa(\sigma)\d \sigma-\int_{0}^{t}\dot{w}(\tau)\frac{w(\tau)}{\lambda(\tau)}\d\tau=\frac 12 \frac{w(0)^2}{\ell_0},$ for every $t\in[0,+\infty)$.
		\end{itemize}
	\end{prop}
\begin{proof}
	Let $(v,\lambda)$ be an energetic evolution, then (o) is satisfied by definition. Now fix $t\in[0,+\infty)$ and choose $\hat{\lambda}=\lambda(t)$ in (S). Then we deduce that $v(t,\cdot)$ minimises the functional $\displaystyle\frac 12 \int_{0}^{\lambda(t)}\dot{\hat{v}}(\sigma)^2\d \sigma$ among all functions $\hat{v}\in H^1(0,\lambda(t))$ such that $\hat{v}(0)=w(t)$ and $\hat{v}(\lambda(t))=0$, and this implies (s1). Choosing now $\hat{v}(x)=w(t)\left(1-\frac{x}{\hat\lambda}\right)\chi_{[0,\hat\lambda]}(x)$ in (S) and exploiting (s1) we get:
	\begin{equation*}
		\frac 12 \frac{w(t)^2}{\lambda(t)}+\int_{\ell_0}^{\lambda(t)}\kappa(\sigma)\d \sigma\le\frac 12 \frac{w(t)^2}{\hat\lambda}+\int_{\ell_0}^{\hat\lambda}\kappa(\sigma)\d \sigma,\quad\quad\text{for every }\hat\lambda\ge\lambda(t).
	\end{equation*}
	This means that the energy $E_t\colon[\lambda(t),+\infty)\to [0,+\infty)$ defined by $\displaystyle E_t(x):=\frac 12\frac{w(t)^2}{x}+\int_{\ell_0}^{x}\kappa(\sigma)\d \sigma$ has a global minimum in $x=\lambda(t)$ and so $\dot{E}_t(\lambda(t))\ge 0$, namely (s2) holds true. Finally (eb) follows by (EB) exploiting (s1).\par
	Assume now that (o), (s1), (s2) and (eb) hold true. To prove that $(v,\lambda)$ is an energetic evolution it is enough to show the validity of (S), being (EB) trivially implied by (eb) and (s1). So let us fix $t\in[0,+\infty)$ and notice that (s2) is equivalent to $\phi_\kappa(\lambda(t))\ge\frac 12 w(t)^2$. By (K1) we hence deduce that $\phi_\kappa(x)\ge\frac 12 w(t)^2$ for every $x\ge\lambda(t)$, i.e. $\dot{E}_t(x)\ge 0$ for every $x\ge\lambda(t)$. This means that $E_t$ has a global minimum in $x=\lambda(t)$ and so we obtain:
	\begin{equation*}
		\frac 12 \frac{w(t)^2}{\lambda(t)}+\int_{\ell_0}^{\lambda(t)}\kappa(\sigma)\d \sigma\le\frac 12 \frac{w(t)^2}{\hat\lambda}+\int_{\ell_0}^{\hat\lambda}\kappa(\sigma)\d \sigma,\quad\quad\text{for every }\hat\lambda\ge\lambda(t),
	\end{equation*}
	which in particular implies (S), since affine functions minimise the potential energy.
\end{proof}
If we do not strenghten the assumptions on the toughness $\kappa$ there is no hope to gain more regularity on $\lambda$, even in the case of a constant loading term $w>0$. Indeed it is enough to consider $\kappa(x)=\frac 12 \frac{w^2}{x^2}$ (in this case $\phi_\kappa$ is constant) to realise that any function satisfying (o) automatically satisfies (s2) and (eb).
\begin{lemma}\label{cont}
	Assume (K2). Then any function $\lambda$ satisfying (o), (s2) and (eb) is continuous.
\end{lemma}
\begin{proof}
	Let us assume by contradiction that there exists a time $\bar t\in[0,+\infty)$ in which $\lambda$ is not continuous, namely $\lambda^-(\bar t\,)<\lambda^+(\bar t\,)$. Here we adopt the convention that $\lambda^-(0)=\lambda(0)=\ell_0$. Exploiting (s2), (eb)  and the continuity of $\kappa$ and $w$ we deduce that:
	\begin{subequations}\label{minus}
		\begin{equation}
		\frac 12\frac{w(\bar t\,)^2}{\lambda^-(\bar t\,)^2}\le \kappa(\lambda^-(\bar t\,)),
		\end{equation}
		\begin{equation}\label{minusb}
			\frac 12 \frac{w(\bar t\,)^2}{\lambda^+(\bar t\,)}+\int_{\ell_0}^{\lambda^+(\bar t\,)}\kappa(\sigma)\d \sigma=\frac 12 \frac{w(\bar t\,)^2}{\lambda^-(\bar t\,)}+\int_{\ell_0}^{\lambda^-(\bar t\,)}\kappa(\sigma)\d \sigma.
		\end{equation}
	\end{subequations}
	By using (K2), from \eqref{minus} we get:
	\begin{align*}
		0&=\int_{\lambda^-(\bar t\,)}^{\lambda^+(\bar t\,)}\kappa(\sigma)\d \sigma-\frac 12 w(\bar t\,)^2\left(\frac{1}{\lambda^-(\bar t\,)}-\frac{1}{\lambda^+(\bar t\,)}\right)=\int_{\lambda^-(\bar t\,)}^{\lambda^+(\bar t\,)}\frac{\phi_\kappa(\sigma)-w(\bar t\,)^2/2}{\sigma^2}\d \sigma\\
		&>\left(\phi_\kappa(\lambda^-(\bar t\,))-\frac 12 w(\bar t\,)^2\right)\int_{\lambda^-(\bar t\,)}^{\lambda^+(\bar t\,)}\frac{1}{\sigma^2}\d \sigma\ge 0.
	\end{align*}
	This leads to a contradiction and hence we conclude.
\end{proof}
\begin{lemma}\label{constant}
		Assume (K2) and let $\lambda$ be a function satisfying (o), (s2) and (eb). If there exists a time $\bar t\in(0,+\infty)$ in which (s2) holds with strict inequality, then $\lambda$ is constant in a neighborhood of $\bar t$.
\end{lemma}
\begin{proof}
	Let us consider the function:\begin{equation*}
			\Phi(t,x):=\frac 12\frac{w(t)^2}{x}+\int_{\ell_0}^{x}\kappa(\sigma)\d \sigma-\int_{0}^{t}\dot{w}(\tau)\frac{w(\tau)}{\lambda(\tau)}\d\tau,\quad\text{ for }(t,x)\in[0,+\infty)\times[\ell_0,+\infty),
	\end{equation*}
	which is continuous on its domain. Moreover the derivative of $\Phi$ in the direction $x$ exists at every point and it is continuous on $[0,+\infty)\times[\ell_0,+\infty)$, being given by:
	\begin{equation*}
		\Phi_x(t,x)=\kappa(x)-\frac 12\frac{w(t)^2}{x^2}.
	\end{equation*}
	Since by assumption $\Phi_x(\bar t,\lambda(\bar t\,))>0$, by continuity we deduce that:
	\begin{equation}\label{greaterc}
		\Phi_x(t,x)\ge m>0,\quad\text{for every }(t,x)\in[a,b]\times[c,d],
	\end{equation} 
	where $[a,b]\times[c,d]\subset(0,+\infty)\times[\ell_0,+\infty)$ is a suitable rectangle containing the point $(\bar t,\lambda(\bar t))$. By continuity of $\lambda$ (given by Lemma~\ref{cont}), we can assume without loss of generality that $\lambda([a,b])\subset[c,d]$. Now we fix $t_1,\,t_2\in[a,b]$, $t_1\le t_2$, and by the mean value Theorem we deduce:
	\begin{equation*}
		\Phi(t_2,\lambda(t_2))-\Phi(t_2,\lambda(t_1))=\Phi_x(t_2,\xi)(\lambda(t_2)-\lambda(t_1)),\quad\text{ for some }\xi\in[\lambda(t_1),\lambda(t_2)]\subset[c,d].
	\end{equation*}
	From this equality, exploiting \eqref{greaterc} and (eb), we get:
	\begin{equation}\label{acest}
		\begin{aligned}
		\lambda(t_2)-\lambda(t_1)&\le\frac 1m\big(	\Phi(t_2,\lambda(t_2))-\Phi(t_2,\lambda(t_1))\big)=\frac 1m\big(	\Phi(t_1,\lambda(t_1))-\Phi(t_2,\lambda(t_1))\big)\\
		&=\frac 1m \left(\frac{1}{2\lambda(t_1)}\big(w(t_1)^2-w(t_2)^2\big)+\int_{t_1}^{t_2}\dot{w}(\tau)\frac{w(\tau)}{\lambda(\tau)}\d\tau\right)\\
		&=\frac 1m \int_{t_1}^{t_2}\dot{w}(\tau)w(\tau)\left(\frac{1}{\lambda(\tau)}-\frac{1}{\lambda(t_1)}\right)\d\tau\\
		&\le \frac{\lambda(t_2)-\lambda(t_1)}{m\ell_0^2}\int_{a}^{b}|\dot{w}(\tau)w(\tau)|\d\tau .
		\end{aligned}
	\end{equation}	
	Since $w$ is absolutely continuous we can also assume that the interval $[a,b]$ is so small that:
	\begin{equation*}
		\frac{1}{m\ell_0^2}\int_{a}^{b}|\dot{w}(\tau)w(\tau)|\d\tau\le \frac 12.
	\end{equation*} 
	From \eqref{acest} we hence deduce that $\lambda(t_2)=\lambda(t_1)$, and so we conclude.
\end{proof}
\begin{rmk}\label{alsoinequality}
	Lemmas~\ref{cont} and \ref{constant} hold true even weakening a bit assumption (eb). It is indeed enough to assume that:
	\begin{equation}\label{nonincr}
	\!\!\!\text{the function}\,\,\,	t\mapsto \frac 12\frac{w(t)^2}{\lambda(t)}+\int_{\ell_0}^{\lambda(t)}\kappa(\sigma)\d \sigma-\int_{0}^{t}\dot{w}(\tau)\frac{w(\tau)}{\lambda(\tau)}\d\tau \,\,\,\text{is nonincreasing in }[0,+\infty).
	\end{equation}
	The only changes in the proofs are in \eqref{minusb} and in the first equality in \eqref{acest}: in this case they become an inequality.
\end{rmk}
We now introduce a notation, already adopted in \cite{Almi} to deal with quasistatic hydraulic fractures: given a continuous function $f\colon[a,b]\to\erre$ we define by $f_*$ the smallest nondecreasing function greater or equal than $f$, namely $f_*(x):=\max\limits_{y\in[a,x]}f(y)$. We refer to \cite{Almi} for its properties, we only want to recall that if $f\in W^{1,p}(a,b)$ for some $p\in[1,+\infty]$, then also $f_*$ belongs to the same Sobolev space and $\dot{f}_*(x)=\dot{f}(x)\chi_{\{f=f_*\}}(x)$ almost everywhere.
\begin{prop}\label{explicitlambda}
	Assume (K2) and let $\lambda$ be a function satisfying (o), (s2) and (eb). Then:
	\begin{equation}\label{solution}
	\lambda(t)=\phi_\kappa^{-1}\left(\max\left\{\frac 12 (w^2)_*(t),\phi_\kappa(\ell_0)\right\}\right),\quad\text{for every }t\in[0,+\infty).
	\end{equation}
\end{prop} 
\begin{proof}
		Let $\lambda$ satisfy (o), (s2) and (eb). By using (s2) we get $\phi_\kappa(\lambda(t))\ge\frac 12 w(t)^2$ for every $t\in[0,+\infty)$, and since the left-hand side is nondecreasing we deduce:
	\begin{equation*}
	\phi_\kappa(\lambda(t))\ge\max\left\{\frac 12 (w^2)_*(t),\phi_\kappa(\ell_0)\right\},\quad\text{for every }t\in[0,+\infty).
	\end{equation*}
	Since by (K2) the function $\phi_\kappa$ is invertible, we finally get that $\lambda(t)\ge\bar{\lambda}(t)$ for every $t\in[0,+\infty)$, where we denoted by $\bar{\lambda}$ the function in the right-hand side of \eqref{solution}.\par 
	Since by Lemma~\ref{cont} we know $\lambda$ is continuous on $[0,+\infty)$ and since by construction the same holds true for $\bar\lambda$, we conclude if we prove that $\lambda(t)=\bar\lambda(t)$ for every $t\in(0,+\infty)$. By contradiction let $\bar t\in(0,+\infty)$ be such that $\lambda(\bar t\,)>\bar\lambda(\bar t\,)$. By (K2) this in particular implies that $\displaystyle\kappa(\lambda(\bar t\,))>\frac 12\frac{w(\bar t\,)^2}{\lambda(\bar t\,)^2}$, and so by Lemma~\ref{constant} we get that $\lambda$ is constant around $\bar t$. Since $\bar\lambda$ is nondecreasing we can repeat this argument getting that $\lambda$ is constant on the whole $[0,\bar t\,]$. This is absurd since it implies:
	\begin{equation*}
		\phi_\kappa(\ell_0)=\phi_\kappa(\lambda(0))=\phi_\kappa(\lambda(\bar t\,))>\phi_\kappa(\bar\lambda(\bar t\,))\ge \phi_\kappa(\ell_0),
	\end{equation*}
	and so we conclude.
\end{proof}
\begin{rmk}\label{alsoinequality2}
	As in Remark~\ref{alsoinequality}, the conclusion of Proposition~\ref{explicitlambda} holds true replacing (eb) by \eqref{nonincr}. This will be useful in the proof of Proposition~\ref{ebquas}.
\end{rmk}
Finally we can state and prove the main result of this Section, regarding the equivalence between the two Definitions~\ref{enev} and \ref{quasev} and about existence and uniqueness of quasistatic evolutions.
\begin{thm}\label{exuniqquas}
	Assume (K3). Then a pair $(v,\lambda)$ is an energetic evolution if and only if it is a quasistatic evolution.\par\noindent 
	In particular, if we in addition assume (KW), the only quasistatic evolution $(\bar{v}, \bar{\lambda})$ is given by:
	\begin{itemize}
		\item $\displaystyle \bar v(t,x)=w(t)\left(1-\frac{x}{\bar\lambda(t)} \right)\chi_{[0,\bar\lambda(t)]}(x),\quad$ for every $(t,x)\in[0,+\infty)\times[0,+\infty)$,
		\item $\displaystyle\bar\lambda(t)=\phi_\kappa^{-1}\left(\max\left\{\frac 12 (w^2)_*(t),\phi_\kappa(\ell_0)\right\}\right),\quad$ for every $t\in[0,+\infty)$.
	\end{itemize}
\end{thm}
\begin{proof}
		Let $(v,\lambda)$ be an energetic evolution. By Proposition~\ref{equiv} we get $v$ satisfies (ii) and $\lambda$ satisfies (o), (s2) and (eb). Moreover by Proposition~\ref{explicitlambda} $\lambda$ is explicitely given by \eqref{solution} and hence by (K3) it is absolutely continuous on $[0,T]$ for every $T>0$, being composition of two nondecreasing absolutely continuous functions. Differentiating (eb) we now conclude that quasistatic Griffith's criterion (iii) holds true and so $(v,\lambda)$ is a quasistatic evolution.\par 
		On the other hand checking that any quasistatic evolution satisfy (o), (s1), (s2) and (eb) is straightforward, and hence by Proposition~\ref{equiv} the other implication is proved.\par 
		Let us now verify that, assuming (KW), the pair $(\bar v,\bar\lambda)$ is actually a quasistatic evolution. By (KW) $\bar\lambda$ is well defined and (i) is fulfilled. The only nontrivial thing to check is the validity of the third condition in the quasistatic Griffith's criterion (iii). We need to prove that for any differentiability point $\bar t\in(0,+\infty)$ of $\bar\lambda$  such that $\dot{\bar\lambda}(\bar t\,)>0$ it holds $\displaystyle \kappa(\bar\lambda(\bar t\,))=\frac 12\frac{w(\bar t\,)^2}{\bar\lambda(\bar t\,)^2}$. From the explicit expression of $\dot{\bar\lambda}$, namely:
		\begin{equation*}
			\dot{\bar\lambda}(t)=\frac{w(t)\dot{w}(t)}{\dot{\phi_\kappa}(\bar\lambda(t))}\chi_{\{w^2=(w^2)_*> 2\phi_\kappa(\ell_0)\}}(t),\quad\quad\text{for a.e. }t\in(0,+\infty),
		\end{equation*}
		we deduce that in $t=\bar t$ we must have $w(\bar t\,)^2=(w^2)_*(\bar t\,)> 2\phi_\kappa(\ell_0)$ and so it holds:
		\begin{equation*}
			\phi_\kappa(\bar\lambda(\bar t\,))=\max\left\{\frac 12 (w^2)_*(\bar t\,),\phi_\kappa(\ell_0)\right\}=\frac 12 w(\bar t\,)^2,
		\end{equation*}
		and we conclude.
\end{proof}

	\section{Energy estimates}\label{sec3}
	In this Section we provide useful energy estimates for the pair of dynamic evolutions $(\ueps,\elleps)$ given by Theorem~\ref{exuniq}. These estimates will be used in the next Section to analyse the limit as $\eps\to 0^+$ of both $\ueps$ and $\elleps$. From now on we always assume that the positive toughness $\kappa$ belongs to $\widetilde{C}^{0,1}([\ell_0,+\infty))$. When needed we will also require the following additional assumptions on the data:
	\begin{itemize}
		\item[(H1)] the families $\{\weps\}_{\eps>0}$, $\{u_0^\eps\}_{\eps>0}$, $\{\eps u_1^\eps\}_{\eps>0}$ are bounded in $\widetilde H^1(0,+\infty)$, $H^1(0,\ell_0)$ and\linebreak$L^2(0,\ell_0)$, respectively;
		\item[(K0)] the function $\kappa$ is not integrable in $[\ell_0,+\infty)$.
	\end{itemize}
\begin{rmk}\label{unifwrmk}
	Whenever we assume (H1), we denote by $\epsn$ a subsequence for which we have:
	\begin{equation}\label{unifw}
		\wepsn\rightharpoonup {w}\text{ in }\widetilde{H}^1(0,+\infty)\quad\quad\text{and}\quad\quad \wepsn\to w\text{ uniformly in }[0,T]\text{ for every }T>0,
	\end{equation}
	for a suitable $w\in\widetilde{H}^1(0,+\infty)$. This sequence can be obtained by weak compactness and Sobolev embedding. By abuse of notation we will not relabel further subsequences.
\end{rmk}
The first step is obtaining an energy bound uniform in $\eps$ from the energy-dissipation balance \eqref{eb}. As one can see, we must deal with the work of the external loading $\mc W^\eps$, so we need to find a way to handle the boundary term $\uepsx(\cdot,0)$. Next Lemma shows how we can recover it via an integration by parts. 
	\begin{lemma}\label{intpartslemma}
		Let the function $h\in C^\infty([0,+\infty))$ satisfy $h(0)=1$, $0\le h(x)\le 1$ for every $x\in[0,+\infty)$ and $h(x)=0$ for every $x\ge\ell_0$. Then the following equality holds true for every $t\in[0,+\infty)$:
		\begin{equation}\label{intparts}
		\!\!\!\!	\begin{aligned}
			&\quad\quad\frac 12\int_{0}^{t}\Big(\eps^2\wepsd(\tau)^2+\uepsx(\tau,0)^2\Big)\d\tau\\
			&=-\frac 12 \int_{0}^{t}\int_{0}^{\ell_0}\dot{h}(\sigma)\Big(\eps^2\uepst(\tau,\sigma)^2+\uepsx(\tau,\sigma)^2\Big)\d\sigma\d\tau-\nu\int_{0}^{t}\int_{0}^{\ell_0}h(\sigma)\eps\uepst(\tau,\sigma)\uepsx(\tau,\sigma)\d\sigma\d\tau\\
			&\quad-\eps\left(\int_{0}^{\ell_0}h(\sigma)\eps\uepst(t,\sigma)\uepsx(t,\sigma)\d\sigma-\int_{0}^{\ell_0}h(\sigma)\eps u^\eps_1(\sigma)\dot{u}^\eps_0(\sigma)\d\sigma\right).
			\end{aligned}
		\end{equation}
	\end{lemma}
	\begin{proof}
		We start with a formal proof, assuming that all the computation we are doing are allowed, and then we make it rigorous via an approximation argument. Performing an integration by parts we deduce:
		\begingroup\allowdisplaybreaks
		\begin{align*}
			&\quad\quad \frac 12\int_{0}^{t}\Big(\eps^2\wepsd(\tau)^2+\uepsx(\tau,0)^2\Big)\d\tau=\frac 12\int_{0}^{t}\Big(\eps^2\uepst(\tau,0)^2+\uepsx(\tau,0)^2\Big)\d\tau\\
			&=-\frac 12\int_{0}^{t}h(0)\Big(\eps^2\uepst(\tau,0)^2+\uepsx(\tau,0)^2\Big)(-1)\d\tau\\
			&=-\frac 12\int_{0}^{t}\int_{0}^{\ell_0}\frac{\partial}{\partial \sigma}\Big[h(\cdot)\Big(\eps^2\uepst(\tau,\cdot)^2+\uepsx(\tau,\cdot)^2\Big)\Big](\sigma)\d\sigma\d\tau\\
			&=-\frac 12\int_{0}^{t}\int_{0}^{\ell_0}\dot{h}(\sigma)\Big(\eps^2\uepst(\tau,\sigma)^2+\uepsx(\tau,\sigma)^2\Big)\d\sigma\d\tau\\
			&\quad-\int_{0}^{t}\int_{0}^{\ell_0}h(\sigma)\Big(\eps^2\uepst(\tau,\sigma)u^\eps_{tx}(\tau,\sigma)+\uepsx(\tau,\sigma)u^\eps_{xx}(\tau,\sigma)\Big)\d\sigma\d\tau=(*).
		\end{align*}
		\endgroup
		Exploiting the fact that $\ueps$ solves problem \eqref{problem2} we hence get:
		\begingroup\allowdisplaybreaks
		\begin{align*}
			(*)&=-\frac 12\int_{0}^{t}\int_{0}^{\ell_0}\dot{h}(\sigma)\Big(\eps^2\uepst(\tau,\sigma)^2+\uepsx(\tau,\sigma)^2\Big)\d\sigma\d\tau\\
			&\quad-\nu \int_{0}^{t}\int_{0}^{\ell_0}h(\sigma)\eps\uepst(\tau,\sigma)\uepsx(\tau,\sigma)\d\sigma\d\tau\\
			&\quad-\eps\int_{0}^{t}\int_{0}^{\ell_0}h(\sigma)\Big(\eps u^\eps_{tt}(\tau,\sigma)\uepsx(\tau,\sigma)+\eps\uepst(\tau,\sigma)u^\eps_{tx}(\tau,\sigma)\Big)\d\sigma\d\tau.
		\end{align*}
		\endgroup
		Now we conclude since it holds:
		\begingroup\allowdisplaybreaks
		\begin{align*}
			&\quad\int_{0}^{t}\int_{0}^{\ell_0}h(\sigma)\Big(\eps u^\eps_{tt}(\tau,\sigma)\uepsx(\tau,\sigma)+\eps\uepst(\tau,\sigma)u^\eps_{tx}(\tau,\sigma)\Big)\d\sigma\d\tau\\
			&=\int_{0}^{\ell_0}h(\sigma)\int_{0}^{t}\frac{\partial}{\partial\tau}\Big[\eps\uepst(\cdot,\sigma)\uepsx(\cdot,\sigma)\Big](\tau)\d\tau\d\sigma\\
			&=\int_{0}^{\ell_0}h(\sigma)\eps\uepst(t,\sigma)\uepsx(t,\sigma)\d\sigma-\int_{0}^{\ell_0}h(\sigma)\eps u^\eps_1(\sigma)\dot{u}^\eps_0(\sigma)\d\sigma.
		\end{align*}
		\endgroup
		All the previous computations are rigorous if $\ueps$ belongs to $\widetilde H^2(\Omega^\eps)$, which is not the case. To overcome this lack of regularity we perform an approximation argument, exploiting Remark~\ref{regularity} and Theorem~\ref{contdependence}.\par 
		Let us consider a sequence $\{{u_0^\eps}_n\}_{n\in\enne}\subset H^2(0,\ell_0)$ such that ${u_0^\eps}_n(0)=u_0^\eps(0)$, ${u_0^\eps}_n(\ell_0)=0$ and converging to $u_0^\eps$ in $H^1(0,\ell_0)$ as $n\to +\infty$; then we pick a sequence $\{w^\eps_n\}_{n\in\enne}\subset\widetilde H^2(0,+\infty)$ such that $w^\eps_n(0)=\weps(0)$ and converging to $w^\eps$ in $\widetilde H^1(0,+\infty)$ as $n\to +\infty$; finally we take another sequence $\{{u_1^\eps}_n\}_{n\in\enne}\subset H^1(0,\ell_0)$ converging to $u_1^\eps$ in $L^2(0,\ell_0)$  as $n\to +\infty$ and satisfying:
		\begin{equation*}
			{u_1^\eps}_n(0)=\dot{w}^\eps_n(0),\quad {u_1^\eps}_n(\ell_0)=\begin{cases}
			-\frac{\text{sign}\big({\dot{u}_{0_n}^\eps}(\ell_0)\big)}{\eps}\sqrt{{\dot{u}_{0_n}^\eps}(\ell_0)^2
				-2\kappa(\ell_0)},&\text{if }{\dot{u}_{0_n}^\eps}(\ell_0)^2>2\kappa(\ell_0),\\
			0,&\text{otherwise}.
			\end{cases}
		\end{equation*}	
		Denoting by $(u^\eps_n,\ell^\eps_n)$ the solution of coupled problem \eqref{problem2}\&\eqref{Griffith} related to these data, we deduce by Remark~\ref{regularity} that $u^\eps_n$ belongs to $H^2(\Omega^\eps_T)$, and so by previous computations \eqref{intparts} holds true for it.  By Theorem~\ref{contdependence} equality \eqref{intparts} passes to the limit as $n\to +\infty$ and hence we conclude.	 
	\end{proof}\noindent
Thanks to previous Lemma we are able to prove the following energy bound:
\begin{prop}\label{energyboundprop}
	Assume (H1). Then for every $T>0$ there exists a positive constant $C_T>0$ such that for every $\eps\in(0,1/2)$ it holds:
	\begin{equation}\label{energybound}
		\mc E^\eps(t)+\mc A^\eps(t)+\int_{\ell_0}^{\elleps(t)}\kappa(\sigma)\d \sigma\le C_T,\quad\quad\text{for every }t\in[0,T],
	\end{equation}
	where $\mc E^\eps$ and $\mc A^\eps$ are the energies defined in \eqref{kinpot} and \eqref{frict}.
\end{prop}
	\begin{proof}
		We fix $T>0$, $t\in[0,T]$, $\eps\in(0,1/2)$ and by using the energy-dissipation balance \eqref{eb} we estimate:
		\begin{align*}
			&\quad\,\,\mc E^\eps(t)+\mc A^\eps(t)+\int_{\ell_0}^{\elleps(t)}\kappa(\sigma)\d \sigma=\mc E^\eps(0)-\mc W^\eps(t)\\
			&\le \mc E^\eps(0)+\frac 12\int_{0}^{t}\dot{w}^\eps(\tau)^2\d\tau +\frac 12\int_{0}^{t}\uepsx(\tau,0)^2\d\tau\\
			&=\mc E^\eps(0)+\frac{1-\eps^2}{2}\int_{0}^{t}\dot{w}^\eps(\tau)^2\d\tau +\frac 12\int_{0}^{t}\Big(\eps^2\dot{w}^\eps(\tau)^2+\uepsx(\tau,0)^2\Big)\d\tau=(*).
		\end{align*} 
		By Lemma~\ref{intpartslemma} and by applying Young's inequality, we can continue the estimate getting:
		\begin{align*}
			(*)&\le\quad\mc E^\eps(0)+\frac{1-\eps^2}{2}\int_{0}^{t}\dot{w}^\eps(\tau)^2\d\tau\\
			&\quad+\left(\max\limits_{x\in[0,\ell_0]}|\dot{h}(x)|+\nu\right)\int_{0}^{t}\mc E^\eps(\tau)\d\tau+\eps\mc E^\eps(t)+\eps\mc E^\eps(0).
		\end{align*}
		We conclude by means of Gr\"onwall Lemma and exploiting (H1).
	\end{proof}\noindent
As an immediate Corollary we have:
\begin{cor}\label{ellbd}
	Assume (H1) and (K0). Then for every $T>0$ there exists a positive constant $L_T>0$ such that $\elleps(T)\le L_T$ for every $\eps\in(0, 1/2)$.
\end{cor}
In order to improve the energy bound given by Proposition~\ref{energyboundprop} we exploit the classical exponential decay of the energy for a solution to the damped wave equation. Following the ideas of \cite{MisGo} we adapt their argument to our model in which the domain of the equation changes in time. For this aim we introduce the modified energy:
	\begin{equation*}
		\widetilde{\mc E}^\eps(t):=\frac 12 \int_{0}^{\elleps(t)}\eps^2\uepst(t,\sigma)^2\d\sigma+\frac 12\int_{0}^{\elleps(t)}\big(\uepsx(t,\sigma)-r^\eps_x(t,\sigma)\big)^2\d\sigma,\quad \text{for }t\in[0,+\infty),
	\end{equation*}
	where $r^\eps(t,x)$ is the affine function connecting the points $(0,\weps(t))$ and $(\elleps(t),0)$, namely: 
	\begin{equation}\label{reps}
		 r^\eps(t,x):=\weps(t)\left(1-\frac{x}{\elleps(t)}\right)\chi_{[0,\elleps(t)]}(x),\quad \text{for }(t,x)\in[0,+\infty)\times[0,+\infty).
	\end{equation}
		The main result of this Section is the following decay estimate:
	\begin{thm}\label{mainestimate}
		Assume (H1) and (K0) and let the parameter $\nu$ be \textbf{positive}. Then for every $T>0$ there exists a constant $C_T>0$ such that for every $t\in[0,T]$ and $\eps\in(0,1/2)$ one has:
		\begin{equation}\label{estimatemain}
		\widetilde{\mc E}^\eps(t)\le 4\widetilde{\mc E}^\eps(0)e^{-m\frac{t}{\eps}}+C_T\int_{0}^{t}\big(\ellepsd(\tau)+\wepsd(\tau)^2+\uepsx(\tau,0)^2+1\big)e^{-m\frac{t-\tau}{\eps}}\d\tau,
		\end{equation}
		where $\displaystyle m=m(\nu,T):=\frac 12 \min\left\{\frac{1}{2\mu^0_T},\,\frac\nu 2,\,\frac{1}{\mu_T^0+\mu^1_T} \right\}>0$ and $\mu^0_T $, $\mu^1_T$ are defined as follows: 
		\begin{equation}\label{muzerouno}
		\mu_T^0:=\frac{L_T}{\pi},\quad\text{ and }\quad \mu_T^1:=\nu\left(\frac{L_T}{\pi}\right)^2,
		\end{equation}
		with $L_T$ given by Corollary~\ref{ellbd}.
	\end{thm}
\begin{rmk}
	Estimate \eqref{estimatemain} actually still holds true for $\nu=0$, but in this case $m=0$ and so the inequality becomes trivial and useless.
\end{rmk}
To prove this Theorem we will need several Lemmas. As before we always assume that $\eps\in(0,1/2)$.
	\begin{lemma}
		Assume (H1). Then for every $T>0$ the modified energy $\widetilde{\mc E}^\eps$ is absolutely continuous on $[0,T]$ and the following inequality holds true for a.e. $t\in(0,T)$:
		\begin{equation}\label{este}
			\dot{\widetilde{\mc E}^\eps}(t)\le -\nu\int_{0}^{\elleps(t)}\eps\uepst(t,\sigma)^2\d\sigma+C_T\big(\ellepsd(t)+\wepsd(t)^2+\uepsx(t,0)^2+1\big),
		\end{equation}
		where $C_T$ is a positive constant depending on $T$ but independent of $\eps$.
	\end{lemma}
	\begin{proof}
		By simple computations one can show that:
		\begin{equation}\label{squares}
		\widetilde{\mc E}^\eps(t)=\mc E^\eps(t)-\frac 12 \frac{\weps(t)^2}{\elleps(t)},\quad \text{for every }t\in[0,+\infty).
		\end{equation}
		Now fix $T>0$. The modified energy $\widetilde{\mc E}^\eps$ is absolutely continuous on $[0,T]$ because by \eqref{squares} it is sum of two absolutely continuous functions (see also Proposition~2.1 in \cite{RivNar}). By \eqref{squares} and the energy-dissipation balance \eqref{eb} we then compute for a.e. $t\in(0,+\infty)$:
		\begin{align*}
			\dot{\widetilde{\mc E}^\eps}(t)&=\dot{\mc E}^\eps(t)-\frac 12\frac{\d}{\d t}\frac{\weps(t)^2}{\elleps(t)}=\\
			&=-\kappa(\elleps(t))\ellepsd(t)-\nu\int_{0}^{\elleps(t)}\eps\uepst(t,\sigma)^2\d\sigma-\wepsd(t)\uepsx(t,0)\\
			&\quad+\frac{\ellepsd(t)}{2}\frac{\weps(t)^2}{\elleps(t)^2}-\wepsd(t)\frac{\weps(t)}{\elleps(t)}.
		\end{align*}
		Recalling that $\elleps(t)\ge\ell_0$ and since by (H1) the family $\{\weps\}_{\eps>0}$ is uniformly equibounded in $[0,T]$ we conclude by means of Young's inequality. 
	\end{proof}\noindent
Always inspired by \cite{MisGo}, for $t\in[0,+\infty)$ we also introduce the auxiliary function:
	\begin{equation*}
		\widetilde{\mc F}^\eps(t):=\int_{0}^{\elleps(t)}\eps^2 \uepst(t,\sigma)\big(\ueps(t,\sigma)-\reps(t,\sigma)\big)\d\sigma+\frac{\nu\eps}{2}\int_{0}^{\elleps(t)}\big(\ueps(t,\sigma)-\reps(t,\sigma)\big)^2\d\sigma.
	\end{equation*}
		\begin{lemma}
		Assume (H1) and (K0). Then for every $T>0$ one has:
		\begin{equation}\label{doublebound}
		-\eps\mu^0_T\widetilde{\mc E}^\eps(t)\le\widetilde{\mc F}^\eps(t)\le\eps(\mu^0_T+\mu^1_T)\widetilde{\mc E}^\eps(t),\quad\text{for every }t\in[0,T],
		\end{equation}
		where $\mu^0_T$ and $\mu^1_T$ have been defined in \eqref{muzerouno}.
	\end{lemma}
	\begin{proof}
	 We fix $t\in[0,T]$ and by means of the sharp Poincarè inequality: 
		\begin{equation}\label{sharppoinc}
		\int_{a}^{b}f(\sigma)^2\d \sigma\le\frac{(b-a)^2}{\pi^2}\int_{a}^{b}\dot f(\sigma)^2\d \sigma,\text{ for every }f\in H^1_0(a,b),
		\end{equation}
		together with Young's inequality we get:
		\begin{align*}
		&\quad\left|\eps^2\int_{0}^{\elleps(t)}\uepst(t,\sigma)\big(\ueps(t,\sigma)-\reps(t,\sigma)\big)\d\sigma\right|\\
		&\le \frac{\eps}{2}\left[\frac{\elleps(t)}{\pi}\int_{0}^{\elleps(t)}\eps^2\uepst(t,\sigma)^2\d\sigma+\frac{\pi}{\elleps(t)}\int_{0}^{\elleps(t)}\big(\ueps(t,\sigma)-\reps(t,\sigma)\big)^2\d\sigma\right]\\
		&\le\eps\frac{\elleps(t)}{\pi}\left[\frac 12\int_{0}^{\elleps(t)}\eps^2\uepst(t,\sigma)^2\d\sigma+\frac 12\int_{0}^{\elleps(t)}\big(\uepsx(t,\sigma)-r^\eps_x(t,\sigma)\big)^2\d\sigma\right]\le \eps\mu^0_T\widetilde{\mc E}^\eps(t).
		\end{align*}
		From the above estimate we hence deduce:
		\begin{align*}
		-\eps\mu^0_T\widetilde{\mc E}^\eps(t)&\le -\left|\eps^2\int_{0}^{\elleps(t)}\uepst(t,\sigma)\big(\ueps(t,\sigma)-\reps(t,\sigma)\big)\d\sigma\right|\le \widetilde{\mc F}^\eps(t)\\
		&\le\eps\mu^0_T\widetilde{\mc E}^\eps(t)+\frac{\eps\nu}{2}\frac{\elleps(t)^2}{\pi^2}\int_{0}^{\elleps(t)}\big(\uepsx(t,\sigma)-r^\eps_x(t,\sigma)\big)^2\d\sigma\\
		&\le \eps(\mu^0_T+\mu^1_T)\widetilde{\mc E}^\eps(t),
		\end{align*}
		and we conclude.
	\end{proof}
	\begin{lemma}
		Assume (H1) and (K0). Then for every $T>0$ the function $\widetilde{\mc F}^\eps$ is absolutely continuous on $[0,T]$ and the following inequality holds true for a.e. $t\in(0,T)$:
		\begin{equation}\label{estf}
			\dot{\widetilde{\mc F}^\eps}(t)\le 2\int_{0}^{\elleps(t)}\eps^2\uepst(t,\sigma)^2\d\sigma-\widetilde{\mc E}^\eps(t)+C_T\eps^2\big(\wepsd(t)^2+\ellepsd(t)^2\big),
		\end{equation}
		where $C_T$ is a positive constant depending on $T$ but independent of $\eps$.
	\end{lemma}
	\begin{proof}
		Fix $T>0$. By exploiting the fact that $\ueps$ solves problem \eqref{problem2} we start formally computing the derivative of $\widetilde{\mc F}^\eps$ at almost every point $t\in(0,T)$:
		\begingroup\allowdisplaybreaks
		\begin{align*}
			\dot{\widetilde{\mc F}^\eps}(t)&=\int_{0}^{\elleps(t)}\eps^2u^\eps_{t}(t,\sigma)\big(\uepst(t,\sigma)-r^\eps_t(t,\sigma)\big)\d\sigma+\int_{0}^{\elleps(t)}\eps^2u^\eps_{tt}(t,\sigma)\big(\ueps(t,\sigma)-\reps(t,\sigma)\big)\d\sigma\\
			&\quad+\nu\eps\int_{0}^{\elleps(t)}\big(\ueps(t,\sigma)-\reps(t,\sigma)\big)\big(\uepst(t,\sigma)-r^\eps_t(t,\sigma)\big)\d\sigma\\
			&=\int_{0}^{\elleps(t)}\!\!\!\!\!\!\eps^2u^\eps_{t}(t,\sigma)\big(\uepst(t,\sigma)-r^\eps_t(t,\sigma)\big)\d\sigma+\int_{0}^{\elleps(t)}\!\!\!\!\!\!\big(\ueps(t,\sigma)-\reps(t,\sigma)\big)\big(u^\eps_{xx}(t,\sigma)-r^\eps_{xx}(t,\sigma)\big)\d\sigma\\
			&\quad-\nu\int_{0}^{\elleps(t)}\eps r^\eps_t(t,\sigma)\big(\ueps(t,\sigma)-\reps(t,\sigma)\big)\d\sigma\\
			&=\int_{0}^{\elleps(t)}\eps^2u^\eps_{t}(t,\sigma)\big(\uepst(t,\sigma)-r^\eps_t(t,\sigma)\big)\d\sigma-\int_{0}^{\elleps(t)}\big(\uepsx(t,\sigma)-r^\eps_x(t,\sigma)\big)^2\d\sigma\\
			&\quad-\nu\int_{0}^{\elleps(t)}\eps r^\eps_t(t,\sigma)\big(\ueps(t,\sigma)-\reps(t,\sigma)\big)\d\sigma.
		\end{align*}
		\endgroup
		By means of an approximation argument similar to the one adopted in the proof of Lemma~\ref{intpartslemma} one deduces that $\widetilde{\mc F}^\eps$ is absolutely continuous on $[0,T]$ and that the formula for $\dot{\widetilde{\mc F}^\eps}$ found with the previous computation is actually true.\par 
		To get \eqref{estf} we use the sharp Poincarè inequality \eqref{sharppoinc} and Young's inequality:
		\begingroup\allowdisplaybreaks
		\begin{align*}
			\dot{\widetilde{\mc F}^\eps}(t)&\le 2\int_{0}^{\elleps(t)}\eps^2u^\eps_{t}(t,\sigma)^2\d\sigma-2\widetilde{\mc E}^\eps(t)+\frac 12 \int_{0}^{\elleps(t)}\eps^2\uepst(t,\sigma)^2\d\sigma+\frac 12 \int_{0}^{\elleps(t)}\eps^2 r^\eps_t(t,\sigma)^2\d\sigma\\
			&\quad+\frac \nu 2\left[\frac{1}{\nu}\frac{\pi^2}{\elleps(t)^2}\int_{0}^{\elleps(t)}\big(\ueps(t,\sigma)-\reps(t,\sigma)\big)^2\d\sigma+\nu\frac{\elleps(t)^2}{\pi^2}\int_{0}^{\elleps(t)}\eps^2 r^\eps_t(t,\sigma)^2\d\sigma\right]\\
			&\le2\int_{0}^{\elleps(t)}\eps^2u^\eps_{t}(t,\sigma)^2\d\sigma-2\widetilde{\mc E}^\eps(t)+\frac 12 \int_{0}^{\elleps(t)}\eps^2\uepst(t,\sigma)^2\d\sigma+\frac 12 \int_{0}^{\elleps(t)}\eps^2 r^\eps_t(t,\sigma)^2\d\sigma\\
			&\quad+\frac 12\int_{0}^{\elleps(t)}\big(\uepsx(t,\sigma)-r^\eps_x(t,\sigma)\big)^2\d\sigma+\frac 12 \left(\frac{\nu\elleps(t)}{\pi}\right)^2\int_{0}^{\elleps(t)}\eps^2 r^\eps_t(t,\sigma)^2\d\sigma\\
			&\le2\int_{0}^{\elleps(t)}\eps^2u^\eps_{t}(t,\sigma)^2\d\sigma-\widetilde{\mc E}^\eps(t)+\frac 12\left(1+\nu\mu^1_T\right)\eps^2\int_{0}^{\elleps(t)} r^\eps_t(t,\sigma)^2\d\sigma.
		\end{align*}
		\endgroup
		To conclude it is enough to use Corollary~\ref{ellbd}, (H1) and to exploit the explicit form of $r^\eps$ given by \eqref{reps} getting:
		\begin{equation*}
			\int_{0}^{\elleps(t)} r^\eps_t(t,\sigma)^2\d\sigma\le C_T\big(\wepsd(t)^2+\ellepsd(t)^2\big).
		\end{equation*}
	\end{proof}\noindent
We are now in a position to prove Theorem~\ref{mainestimate}:
\begin{proof}[Proof of Theorem \ref{mainestimate}]
	We fix $T>0$ and we introduce the Lyapunov function:
	\begin{equation*}
		\widetilde{\mc D}^\eps(t):=\widetilde{\mc E}^\eps(t)+\frac{2m}{\eps} \widetilde{\mc F}^\eps(t),\quad\text{for }t\in[0,T].
	\end{equation*}
	From \eqref{doublebound} we easily infer:
	\begin{equation*}
		\big(1-2m\mu^0_T\big)\widetilde{\mc E}^\eps(t)\le\widetilde{\mc D}^\eps(t)\le \big(1+2m(\mu^0_T+\mu^1_T)\big)\widetilde{\mc E}^\eps(t),\quad \text{for every }t\in[0,T],
	\end{equation*}
	and so in particular by definition of $m$ we deduce:
	\begin{equation}\label{ed}
		\frac 12\widetilde{\mc E}^\eps(t)\le\widetilde{\mc D}^\eps(t)\le 2\widetilde{\mc E}^\eps(t), \text{ for every }t\in[0,T].
	\end{equation}
	Moreover we can estimate the derivative of $\widetilde{\mc D}^\eps$ for a.e. $t\in(0,T)$ by using \eqref{este} and \eqref{estf} and recalling that $\eps\ellepsd(t)<1$ and that $4m\le\nu$:
	\begin{align*}
		\dot{\widetilde{\mc D}^\eps}(t)&=\dot{\widetilde{\mc E}^\eps}(t)+\frac{2m}{\eps}\dot{\widetilde{\mc F}^\eps}(t)\\
		&\le-\big(\nu-4m\big)\int_{0}^{\elleps(t)}\eps\uepst(t,\sigma)^2\d\sigma-\frac{2m}{\eps} \widetilde{\mc E}^\eps(t)+C_T\big(\ellepsd(t)+\wepsd(t)^2+\uepsx(t,0)^2+1\big)\\
		&\le-\frac{2m}{\eps} \widetilde{\mc E}^\eps(t)+C_T\big(\ellepsd(t)+\wepsd(t)^2+\uepsx(t,0)^2+1\big).
	\end{align*}
	By \eqref{ed} we hence deduce:
	\begin{equation*}
		\dot{\widetilde{\mc D}^\eps}(t)\le -\frac{m}{\eps}\widetilde{\mc D}^\eps(t)+C_T\big(\ellepsd(t)+\wepsd(t)^2+\uepsx(t,0)^2+1\big),\quad\text{for a.e. }t\in(0,T),
	\end{equation*}
	from which for every $t\in[0,T]$ we get:
	\begin{equation*}
		\widetilde{\mc D}^\eps(t)\le\widetilde{\mc D}^\eps(0)e^{-m\frac{t}{\eps}}+C_T\int_{0}^{t}\big(\ellepsd(\tau)+\wepsd(\tau)^2+\uepsx(\tau,0)^2+1\big)e^{-m\frac{t-\tau}{\eps}}\d\tau.
	\end{equation*}
	We conclude by using again \eqref{ed}.
\end{proof}

	\section{Quasistatic limit}\label{sec4}
	In this Section we show how, thanks to the estimates of Section \ref{sec3}, dynamic evolutions $\left(\ueps,\elleps\right)$ converge to a quasistatic one as $\eps\to 0^+$, except for a possible initial jump due to an excess of initial potential energy. The rigorous result is stated in Theorem~\ref{finalthm}.
	Also in this Section we assume that $\kappa$ belongs to $\widetilde{C}^{0,1}([\ell_0,+\infty))$.
	\subsection{Extraction of convergent subsequences}
	We first prove that the sequence of debonding fronts $\elleps$ admits a pointwise convergent subsequence.
	\begin{prop}\label{ellconv}
		Assume (H1) and (K0). Then there exists a subsequence $\epsn\searrow 0$ and there exists a nondecreasing function $\ell\colon[0,+\infty)\to[\ell_0,+\infty)$ such that
		\begin{equation*}
		\lim\limits_{n\to+\infty}\ellepsn(t)=\ell(t),\quad\text{for every }t\in[0,+\infty).
		\end{equation*}
	\end{prop}
	\begin{proof}
		The result follows by Corollary~\ref{ellbd} and by a simple application of the classical Helly's selection principle.
	\end{proof}
In order to deal with the convergence of the vertical displacements $\ueps$ we exploit the energy decay \eqref{mainestimate}:
	\begin{prop}\label{energydecay}
		Assume (H1) and (K0) and let $\nu$ be \textbf{positive}. Then for every $T>0$ the modified energy $\widetilde{\mc E}^\eps$ converges to $0$ in $L^1(0,T)$ when $\eps\to 0^+$. Thus there exists a subsequence $\epsn\searrow 0$ such that:
		\begin{equation*}
		\lim\limits_{n\to+\infty}\widetilde{\mc E}^\epsn(t)=0,\quad\text{for almost every }t\in(0,+\infty).
		\end{equation*}
	\end{prop}
	\begin{proof}
		We fix $T>0$. Theorem~\ref{mainestimate} ensures that:
		\begin{equation*}
		\widetilde{\mc E}^\eps(t)\le 4\widetilde{\mc E}^\eps(0)e^{-m\frac{t}{\eps}}+C_T(\rho^\eps*\eta^\eps)(t),\quad \text{for every }t\in[0,T],
		\end{equation*} 
		where the symbol $*$ denotes the convolution product and for a.e. $t\in\erre$ we define:
		\begin{align*}
			&\rho^\eps(t):=\big(\ellepsd(t)+\wepsd(t)^2+\uepsx(t,0)^2+1\big)\chi_{[0,T]}(t),\\
			&\eta^\eps(t):=e^{-m\frac t\eps}\chi_{[0,+\infty)}(t).
		\end{align*} 
		Furthermore by \eqref{squares} and (H1) we get that $\widetilde{\mc E}^\eps(0)$ is uniformly bounded in $\eps$, and so by classical properties of convolutions we estimate:
		\begin{align*}
			\Vert\widetilde{\mc E}^\eps\Vert_{L^1(0,T)}&\le C\int_{0}^{+\infty}e^{-m\frac{\tau}{\eps}}\d\tau+C_T\Vert\rho^\eps*\eta^\eps\Vert_{L^1(\erre)}\\
			&\le C \frac\eps m+C_T\Vert\rho^\eps\Vert_{L^1(\erre)}\Vert\eta^\eps\Vert_{L^1(\erre)}=\frac \eps m\big(C+C_T\Vert\rho^\eps\Vert_{L^1(\erre)}\big).
		\end{align*}
		Now we bound the $L^1$-norm of $\rho^\eps$ by means of (H1), (K0) and recalling that by Lemma~\ref{intpartslemma} and Proposition~\ref{energyboundprop} we know that $\Vert\uepsx(\cdot,0)\Vert_{L^2(0,T)}$ is uniformly bounded with respect to $\eps$:
		\begin{align*}
			\Vert\rho^\eps\Vert_{L^1(\erre)}&=\elleps(T)-\ell_0+\Vert\wepsd\Vert^2_{L^2(0,T)}+\Vert\uepsx(\cdot,0)\Vert^2_{L^2(0,T)}+T\le C_T.
		\end{align*}
		Thus we deduce that $\widetilde{\mc E}^\eps\to 0$ in $L^1(0,T)$ when $\eps\to 0^+$ and so we conclude by using a diagonal argument.
	\end{proof}
	Similarly to what we did in Lemma~\ref{intpartslemma} we need to understand the behaviour of $\uepsx(\cdot,0)$ when $\eps\to 0^+$ before carrying on the analysis of the convergence of $\ueps$.
	\begin{lemma}
		Let the function $h$ be as in Lemma~\ref{intpartslemma}. Then the following equality holds true for every $t\in[0,+\infty)$:
		\begin{equation}\label{intparts2}
		\begin{aligned}
		&\quad\quad\frac 12\int_{0}^{t}\Big(\eps^2\wepsd(\tau)^2+\big(\uepsx(\tau,0)-r^\eps_x(\tau,0)\big)^2\Big)\d\tau\\
		&=-\frac 12 \int_{0}^{t}\int_{0}^{\ell_0}\dot{h}(\sigma)\Big(\eps^2\uepst(\tau,\sigma)^2+\big(\uepsx(\tau,\sigma)-r^\eps_x(\tau,\sigma)\big)^2\Big)\d\sigma\d\tau\\
		&\quad-\nu\int_{0}^{t}\int_{0}^{\ell_0}h(\sigma)\eps\uepst(\tau,\sigma)\big(\uepsx(\tau,\sigma)-r^\eps_x(\tau,\sigma)\big)\d\sigma\d\tau\\
		&\quad-\eps\left(\int_{0}^{\ell_0}h(\sigma)\eps\uepst(t,\sigma)\uepsx(t,\sigma)\d\sigma-\int_{0}^{\ell_0}h(\sigma)\eps u^\eps_1(\sigma)\dot{u}^\eps_0(\sigma)\d\sigma\right)\\
		&\quad-\eps\int_{0}^{\ell_0}h(\sigma)\left(\frac{\weps(t)}{\elleps(t)}\eps\uepst(t,\sigma)-\frac{\weps(0)}{\ell_0}\eps u_1^\eps(\sigma)\right)\d\sigma\\
		&\quad+\eps\int_{0}^{t}\int_{0}^{\ell_0}h(\sigma)\eps\uepst(\tau,\sigma)\frac{\wepsd(\tau)\elleps(\tau)-\weps(\tau)\ellepsd(\tau)}{\elleps(\tau)^2}\d\sigma\d\tau.
		\end{aligned}
		\end{equation}
	\end{lemma}
\begin{proof}
	The proof follows by using exactly the same argument adopted in Lemma~\ref{intpartslemma}, recalling the explicit formula of the affine function $r^\eps$ given by \eqref{reps}.
\end{proof}
\begin{cor}\label{bordo}
	Assume (H1) and (K0) and let $\nu>0$. Then for every $T>0$ one has:
	\begin{equation*}
		\uepsx(\cdot,0)-r^\eps_x(\cdot,0)\to 0,\quad\text{in }L^2(0,T)\text{ as }\eps\to 0^+.
	\end{equation*} 
	Moreover, considering the subsequence $\epsn$ given by \eqref{unifw} and Proposition~\ref{ellconv}, one gets:
	\begin{equation}\label{convboundary}
	\uepsnx(\cdot,0)\to -\frac{w}{\ell},\quad\text{in }L^2(0,T)\text{ as }n\to +\infty,
	\end{equation}
	where $w$ is given by \eqref{unifw} and $\ell$ is the function obtained in Proposition~\ref{ellconv}.
\end{cor}
\begin{proof}
	We fix $T>0$ and we simply estimate by using \eqref{intparts2} and recalling that by (H1) the family $\{\weps\}_{\eps>0}$ is uniformly equibounded in $[0,T]$:
	\begin{align*}
		&\quad\int_{0}^{T}\big(\uepsx(\tau,0)-r^\eps_x(\tau,0)\big)^2\d\tau\\
		&\le C_T\left[\int_{0}^{T}\widetilde{\mc E}^\eps(\tau)\d\tau +\eps\left(\mc E^\eps(t)+\mc E^\eps(0)+\int_{0}^{T}\wepsd(\tau)^2\d\tau+1+\int_{0}^{T}\!\!\!\eps\ellepsd(\tau)\int_{0}^{\elleps(\tau)}\!\!\!\!\!\!\!|\uepst(\tau,\sigma)|\d\sigma\d\tau\right)\right].
	\end{align*}
	By Hölder's inequality and since $\eps\ellepsd(t)<1$ almost everywhere we then deduce:
	\begin{equation*}
		\int_{0}^{T}\eps\ellepsd(\tau)\int_{0}^{\elleps(\tau)}|\uepst(\tau,\sigma)|\d\sigma\d\tau\le \sqrt{TL_T}\left(\int_{0}^{T}\int_{0}^{\elleps(\tau)}\uepst(\tau,\sigma)^2\d\sigma\d\tau\right)^\frac 12=\sqrt{\frac{TL_T}{\eps\nu}}\mc A^\eps(T)^\frac 12.
	\end{equation*}
	By means of Proposition~\ref{energyboundprop} we hence obtain:
	\begin{equation*}
	\quad\int_{0}^{T}\big(\uepsx(\tau,0)-r^\eps_x(\tau,0)\big)^2\d\tau\le C_T\left[\int_{0}^{T}\widetilde{\mc E}^\eps(\tau)\d\tau +\eps\left(\Vert\wepsd\Vert^2_{L^2(0,T)}+1\right)+\sqrt{\eps}\right].
	\end{equation*}
	We conclude by using (H1) and Proposition~\ref{energydecay}.\par 
	The proof of \eqref{convboundary} trivially follows by triangular inequality, recalling that by \eqref{reps} we know that $\reps_x(t,0)=-\frac{\weps(t)}{\elleps(t)}$ for every $t\in[0,+\infty)$.
\end{proof}
We are now in a position to state our first result about the convergence of $\ueps$ to the proper affine function.
\begin{thm}\label{convu}
	Assume (H1), (K0), $\nu>0$ and let $\epsn$ be the subsequence given by \eqref{unifw}, Propositions~\ref{ellconv} and \ref{energydecay}. Let $\ell$ be the nondecreasing function obtained in Proposition~\ref{ellconv}. Then as $n\to+\infty$ one has:
	\begin{itemize}
		\item $\epsn\uepsnt(t,\cdot)\to 0$ strongly in $L^2(0,+\infty)$,$\quad\quad\,$ for every $t\in(0,+\infty)\diff J_\ell$,
		\item $\uepsn(t,\cdot)\to u(t,\cdot)$ strongly in $H^1(0,+\infty)$,$\quad$ for every $t\in(0,+\infty)\diff J_\ell$,
	\end{itemize}
	where $J_\ell$ is the jump set of $\ell$ and:
	\begin{equation*}
	u(t,x):=w(t)\left(1-\frac{x}{\ell(t)}\right)\chi_{[0,\ell(t)]}(x),\quad \text{for }(t,x)\in[0,+\infty)\times[0,+\infty),
	\end{equation*}
	with $w$ given by \eqref{unifw}.
\end{thm}
\begin{proof}
	By \eqref{unifw} and by Proposition~\ref{ellconv} it is easy to see that for every $t\in[0,+\infty)$ one has $\repsn(t,\cdot)\to u(t,\cdot)$ strongly in $H^1(0,+\infty)$ as $n\to +\infty$, thus we deduce:
	\begin{align*}
		&\quad\,\,\Vert\epsn\uepsnt(t,\cdot)\Vert^2_{L^2(0,+\infty)}+\Vert\uepsn(t,\cdot)-u(t,\cdot)\Vert^2_{H^1(0,+\infty)}\\
		&\le C\left( \Vert\epsn\uepsnt(t,\cdot)\Vert^2_{L^2(0,+\infty)}+\Vert\uepsn(t,\cdot)-\repsn(t,\cdot)\Vert^2_{H^1(0,+\infty)}+\Vert\repsn(t,\cdot)-u(t,\cdot)\Vert^2_{H^1(0,+\infty)}\right)\\
		&\le C\left( \Vert\epsn\uepsnt(t,\cdot)\Vert^2_{L^2(0,+\infty)}+\Vert\uepsnx(t,\cdot)-r^\epsn_x(t,\cdot)\Vert^2_{L^2(0,+\infty)}+\Vert\repsn(t,\cdot)-u(t,\cdot)\Vert^2_{H^1(0,+\infty)}\right)\\
		&=C\left(\widetilde{\mc E}^\epsn(t)+\Vert\repsn(t,\cdot)-u(t,\cdot)\Vert^2_{H^1(0,+\infty)}\right),
	\end{align*}
	where we used Poincarè inequality. \par 
	To conclude it is enough to show that $\lim\limits_{n\to+\infty}\widetilde{\mc E}^\epsn(t)=0$ for every $t\in(0,+\infty)\diff J_\ell$. By \eqref{unifw} and \eqref{squares} this is equivalent to prove that:
	\begin{equation}\label{goal}
		\lim\limits_{n\to+\infty}{\mc E}^\epsn(t)=\frac 12 \frac{w(t)^2}{\ell(t)},\quad\text{ for every } t\in(0,+\infty)\diff J_\ell.
	\end{equation}
 By Proposition~\ref{energydecay} we know that \eqref{goal} holds true for a.e. $t\in(0,+\infty)$. To improve the result we then fix $t\in(0,+\infty)\diff J_\ell$ and we consider two sequences $\{s_j\}_{j\in\enne}$ and $\{t_j\}_{j\in\enne}$ such that $0<s_j\le t\le t_j$, the limit in \eqref{goal} holds true for $s_j$ and $t_j$ for every $j\in\enne$ and $s_j\nearrow t$, $t_j\searrow t$ as $j\to+\infty$. By the energy-dissipation balance \eqref{eb} we hence get:
 \begin{equation*}
 	\mc E^\epsn(t_j)+\int_{t}^{t_j}\wepsnd(\tau)\uepsnx(\tau,0)\d\tau\le\mc E^\epsn(t)\le	\mc E^\epsn(s_j)+\int_{t}^{s_j}\wepsnd(\tau)\uepsnx(\tau,0)\d\tau.
 \end{equation*}
 Passing to the limit as $n\to+\infty$ and exploiting Corollary~\ref{bordo} together with \eqref{unifw} we deduce:
 \begin{equation*}
 	\frac 12 \frac{w(t_j)^2}{\ell(t_j)}-\int_{t}^{t_j}\dot{w}(\tau)\frac{w(\tau)}{\ell(\tau)}\d\tau\le\liminf\limits_{n\to+\infty}\mc E^\epsn(t)\le\limsup\limits_{n\to+\infty}\mc E^\epsn(t)\le\frac 12 \frac{w(s_j)^2}{\ell(s_j)}-\int_{t}^{s_j}\dot{w}(\tau)\frac{w(\tau)}{\ell(\tau)}\d\tau.
 \end{equation*}
 Passing now to the limit as $j\to+\infty$, recalling that $t$ is a continuity point of $\ell$, we finally obtain:
 \begin{equation*}
 	\frac 12 \frac{w(t)^2}{\ell(t)}\le\liminf\limits_{n\to+\infty}\mc E^\epsn(t)\le\limsup\limits_{n\to+\infty}\mc E^\epsn(t)\le\frac 12 \frac{w(t)^2}{\ell(t)},
 \end{equation*}
 and so we conclude.
\end{proof}
We want to highlight that the viscous term in the wave equation forces the kinetic energy to vanish when $\eps\to0^+$. Indeed this phenomenon does not happen in \cite{LazNar}, where on the contrary the presence of a persistent kinetic energy due to lack of friction is the main reason why the convergence of $\ueps$ to an affine function occurs only in a weak sense (see Theorem~3.5 in \cite{LazNar}) and the limit pair $(u,\ell)$ fails to be a quasistatic evolution.
\subsection{Characterisation of the limit debonding front}
Our aim now is to understand if the limit function $\ell$ solves quasistatic Griffith's criterion. We thus need to pass to the limit in the dynamic Griffith's criterion \eqref{Griffith}. Next Proposition deals with the stability condition.
\begin{prop}\label{stabprop}
	Assume (H1), (K0), $\nu>0$ and let $\ell$ be the nondecreasing function obtained in Proposition~\ref{ellconv}. Then for every $0\le s\le t$ one has:
	\begin{equation*}
		 \frac 12\int_{s}^{t}\frac{w(\tau)^2}{\ell(\tau)^2}\d\tau\le\int_{s}^{t}\kappa(\ell(\tau))\d\tau,
		 \end{equation*}
	where $w$ is given by \eqref{unifw}.\\ 
	In particular the following inequalities hold true:
	\begin{subequations}
		\begin{equation}\label{stablimplus}
		\frac 12\frac{w(t)^2}{\ell^+(t)^2}\le \kappa(\ell^+(t)),\quad\text{for every }t\in[0,+\infty),
		\end{equation}
		\begin{equation}\label{stablimit}
		\frac 12\frac{w(t)^2}{\ell^-(t)^2}\le \kappa(\ell^-(t)),\quad\text{for every }t\in(0,+\infty),
		\end{equation}
	\end{subequations}
where $\ell^+$ and $\ell^-$ are the right limit and the left limit of $\ell$, respectively.
\end{prop}
	\begin{proof}
		Let $\epsn$ be the subsequence given by \eqref{unifw} and Proposition~\ref{ellconv}. By \eqref{expgrif} we know that for a.e. $t\in(0,+\infty)$ one has:
		\begin{equation}\label{dynenrate}
			G^\epsn_{\epsn\ellepsnd(t)}(t)=2\frac{1-\epsn\ellepsnd(t)}{1+\epsn\ellepsnd(t)}F^\epsn(t-\epsn\ellepsn(t))^2=2\frac{\dot\varphi^\epsn(t)}{\dot\psi^\epsn(t)}F^\epsn(\varphi^\epsn(t))^2,
		\end{equation}
		where we introduced the function:
		\begin{equation*}
			F^\epsn(\sigma)=\dot{f}^\epsn(\sigma)+\nu g^\epsn[\uepsnt](\sigma),\quad\text{for a.e. }\sigma\in(-\epsn\ell_0,\varphi^\epsn(+\infty)).
		\end{equation*}
		Here we adopt the notation $\varphi^\epsn(+\infty)=\lim\limits_{t\to+\infty}\varphi^\epsn(t)$, which exists since $\varphi^\epsn$ is strictly increasing. We want also to remark that $\varphi^\epsn(+\infty)>0$ for $n$ large enough (actually it diverges to $+\infty$ as $n\to+\infty$), indeed $\varphi^\epsn$ converges locally uniformly to the identity map as $n\to+\infty$ by Corollary~\ref{ellbd}. By means of \eqref{expux} and of the explicit form of $\dot f^\epsn$ and $g^\epsn[\uepsnt]$ in $(-\epsn\ell_0,0)$ we deduce that:
		\begin{equation*}
		F^\epsn(\sigma)\!=\!\begin{cases}
		\displaystyle\frac 12 \epsn\wepsnd(\sigma)-\frac 12\uepsnx(\sigma,0)+\nu\Big(g^\epsn[\uepsnt](\sigma)-\frac 12H^\epsn[\uepsnt]_x(\sigma,0)\Big),&\!\!\text{if }\sigma\in(0,\varphi^\epsn(+\infty)),\\
		\displaystyle\frac 12 \epsn u_1^\epsn\left({-}\frac\sigma\epsn\right)-\frac 12  \dot u_0^\epsn\left({-}\frac\sigma\epsn\right)-\frac \nu 2\int_{0}^{{\varphi^\epsn}^{-1}(\sigma)}\!\!\!\!\!\!\!\!\!\!\!\!\uepsnt\left(\tau,\frac{\tau-\sigma}{\epsn}\right)\d\tau,&\!\!\text{if }\sigma\in(-\epsn\ell_0,0).
		\end{cases}
		\end{equation*}
		Thus, thanks to \eqref{magic}, we obtain:
		\begin{equation}\label{F}
		\!\!F^\epsn(\sigma)\!=\!\begin{cases}
		\displaystyle\frac 12 \epsn\wepsnd(\sigma)-\frac 12\uepsnx(\sigma,0)-\frac \nu 2\int_{\sigma}^{{\varphi^\epsn}^{-1}(\sigma)}\!\!\!\!\!\!\!\!\!\!\!\!\uepsnt\left(\tau,\frac{\tau-\sigma}{\epsn}\right)\!\d\tau,&\!\!\!\!\!\!\!\!\!\!\!\text{if }\sigma\in(0,\varphi^\epsn(+\infty)),\\
		\displaystyle\frac 12 \epsn u_1^\epsn\!\left({-}\frac\sigma\epsn\right)\!-\!\frac 12 \dot u_0^\epsn\!\left({-}\frac\sigma\epsn\right)\!-\!\frac \nu 2\int_{0}^{{\varphi^\epsn}^{-1}(\sigma)}\!\!\!\!\!\!\!\!\!\!\!\!\uepsnt\left(\tau,\frac{\tau-\sigma}{\epsn}\right)\!\d\tau,&\!\!\text{if }\sigma\in(-\epsn\ell_0,0).
		\end{cases}
		\end{equation} 
		By the stability condition in dynamic Griffith's criterion \eqref{Griffith} we hence deduce that for every $0\le s\le t$ one has:
		\begin{align*}
			\int_{s}^{t}\kappa(\ellepsn(\tau))\d\tau&\ge\int_{s}^{t}G^\epsn_{\epsn\ellepsnd(\tau)}(\tau)\d\tau=2\int_{s}^{t}\frac{\dot\varphi^\epsn(\tau)}{\dot\psi^\epsn(\tau)}F^\epsn(\varphi^\epsn(\tau))^2\d\tau\\
			&=\int_{\varphi^\epsn(s)}^{\varphi^\epsn(t)}\frac{2}{\dot\psi^\epsn({\varphi^\epsn}^{-1}(\sigma))}F^\epsn(\sigma)^2\d\sigma=:I^\epsn(s,t).
		\end{align*}
		Thus, by dominated convergence we infer:
		\begin{equation*}
			\int_{s}^{t}\kappa(\ell(\tau))\d\tau\ge\limsup\limits_{n\to+\infty}I^\epsn(s,t).
		\end{equation*}
		We actually prove that the limit in the right-hand side exists and it holds:
		\begin{equation}\label{aim}
			\lim\limits_{n\to+\infty}I^\epsn(s,t)=\frac 12\int_{s}^{t}\frac{w(\tau)^2}{\ell(\tau)^2}\d\tau.
		\end{equation}
		If \eqref{aim} is true, then we conclude; to prove it we reason as follows. We first assume $s>0$, so that $\varphi^\epsn(s)>0$ (for $n$ large enough) and we can write:
		\begin{equation*}
			I^\epsn(s,t)=\frac 12\int_{0}^{t}\frac{\chi_{[\varphi^\epsn(s),\varphi^\epsn(t)]}(\sigma)}{\dot\psi^\epsn({\varphi^\epsn}^{-1}(\sigma))}\Big(2F^\epsn(\sigma)\Big)^2\chi_{[0,\varphi^\epsn(t)]}(\sigma)\d\sigma.
		\end{equation*}
		By means of the properties of $\varphi^\epsn$ and $\psi^\epsn$, see \eqref{phipsidef} and the subsequent discussion, and recalling Corollary~\ref{ellbd} it is easy to see that the function $\displaystyle a^\epsn(\sigma):=\frac{\chi_{[\varphi^\epsn(s),\varphi^\epsn(t)]}(\sigma)}{\dot\psi^\epsn({\varphi^\epsn}^{-1}(\sigma))}$ satisfies $\Vert a^\epsn\Vert_{L^\infty(0,t)}\le 1$ and $a^\epsn\to \chi_{[s,t]}$ in $L^1(0,t)$ as $n\to +\infty$. So we conclude if we prove that:
		\begin{equation}\label{aim2}
			2F^\epsn\chi_{[0,\varphi^\epsn(t)]}\to \frac{w}{\ell}, \text{ in }L^2(0,t)\text{ as }n\to+\infty,
		\end{equation}
		since the function $w/\ell$ belongs to $L^\infty(0,t)$. To prove \eqref{aim2} we estimate:
		\begin{align*}
			&\quad\left\Vert2F^\epsn\chi_{[0,\varphi^\epsn(t)]}-\frac w\ell\right\Vert_{L^2(0,t)}\\
			&\le \epsn\Vert\wepsnd\Vert_{L^2(0,t)}+\left\Vert\uepsnx(\cdot,0)+\frac w\ell\right\Vert_{L^2(0,t)}\!\!\!\!+\nu\left(\int_{0}^{\varphi^\epsn(t)}\left(\int_{\sigma}^{{\varphi^\epsn}^{-1}(\sigma)}\!\!\!\!\!\!\!\!\!\!\!\uepsnt\left(\tau,\frac{\tau-\sigma}{\epsn}\right)\!\d\tau\right)^2\!\!\!\!\d\sigma\right)^\frac 12+C\epsn.
		\end{align*}
		By (H1) and \eqref{convboundary} the first and the second term go to zero as $n\to +\infty$. For the third one we continue the estimate:
		\begingroup\allowdisplaybreaks
		\begin{align}\label{estfriction}
			&\quad\int_{0}^{\varphi^\epsn(t)}\left(\int_{\sigma}^{{\varphi^\epsn}^{-1}(\sigma)}\!\!\!\!\!\!\!\!\!\!\!\uepsnt\left(\tau,\frac{\tau-\sigma}{\epsn}\right)\!\d\tau\right)^2\!\!\!\!\d\sigma\nonumber\\
			&\le\int_{0}^{\varphi^\epsn(t)}\!\!\!\!\!\!\!\!\!\!\!\!\!\!\!\!\big({\varphi^\epsn}^{-1}(\sigma)-\sigma\big)\!\!\int_{\sigma}^{{\varphi^\epsn}^{-1}(\sigma)}\!\!\!\!\!\!\!\!\!\!\!\uepsnt\left(\tau,\frac{\tau-\sigma}{\epsn}\right)^2\!\!\!\!\!\d\tau\d\sigma=\int_{0}^{\varphi^\epsn(t)}\!\!\!\!\!\!\!\!\!\!\!\!\!\!\epsn\ellepsn({\varphi^\epsn}^{-1}(\sigma))\int_{\sigma}^{{\varphi^\epsn}^{-1}(\sigma)}\!\!\!\!\!\!\!\!\!\!\!\uepsnt\left(\tau,\frac{\tau-\sigma}{\epsn}\right)^2\!\!\!\!\!\d\tau\d\sigma\nonumber\\
			&\le C_t\int_{0}^{\varphi^\epsn(t)}\int_{0}^{t}\epsn\uepsnt\left(\tau,\frac{\tau-\sigma}{\epsn}\right)^2\chi_{[\sigma,{\varphi^\epsn}^{-1}(\sigma)]}(\tau)\d\tau\d\sigma\\
			&=C_t\int_{0}^{t}\int_{0}^{\varphi^\epsn(t)}\epsn\uepsnt\left(\tau,\frac{\tau-\sigma}{\epsn}\right)^2\chi_{[\sigma,{\varphi^\epsn}^{-1}(\sigma)]}(\tau)\d\sigma\d\tau\nonumber\\
			&\le C_t\int_{0}^{t}\int_{0}^{\ellepsn(\tau)}\epsn^2\uepsnt(\tau,\sigma)^2\d\sigma\d\tau=\epsn \frac{C_t}{\nu}\mc A^\epsn(t),\nonumber
		\end{align}
		\endgroup
		which goes to zero by \eqref{energybound}, and we conclude in the case $s>0$.\par 
		If instead $s=0$ we can write:
		\begin{align*}
			I^\epsn(0,t)=&\frac 12\int_{-\epsn\ell_0}^{0}\frac{1}{\dot\psi^\epsn({\varphi^\epsn}^{-1}(\sigma))}\left[\epsn u_1^\epsn\!\left({-}\frac\sigma\epsn\right)\!-\! \dot u_0^\epsn\!\left({-}\frac\sigma\epsn\right)\!-\! \nu \int_{0}^{{\varphi^\epsn}^{-1}(\sigma)}\!\!\!\!\!\!\!\!\!\!\!\!\uepsnt\left(\tau,\frac{\tau-\sigma}{\epsn}\right)\!\d\tau\right]^2\!\!\!\d\sigma\\
			+&\frac 12\int_{0}^{t}\frac{1}{\dot\psi^\epsn({\varphi^\epsn}^{-1}(\sigma))}\Big(2F^\epsn(\sigma)\Big)^2\chi_{[0,\varphi^\epsn(t)]}(\sigma)\d\sigma.
		\end{align*}
		Reasoning as before one can show that the second term goes to $\displaystyle\frac 12\int_{0}^{t}\frac{w(\tau)^2}{\ell(\tau)^2}\d\tau$ as $n\to+\infty$, so we conclude if we prove that the first one, denoted by $J^\epsn$, vanishes in the limit. To this aim we estimate:
		\begin{align*}
			J^\epsn&\le C\int_{-\epsn\ell_0}^{0}\left[\epsn^2u_1^\epsn\left({-}\frac\sigma\epsn\right)^2+\dot u_0^\epsn\left({-}\frac\sigma\epsn\right)^2+\nu\epsn\int_{0}^{{\varphi^\epsn}^{-1}(\sigma)}\!\!\!\!\!\!\!\!\!\!\!\!\uepsnt\left(\tau,\frac{\tau-\sigma}{\epsn}\right)^2\!\d\tau\right]\d\sigma\\
			&\le \epsn C\left(\Vert\epsn u_1^\epsn\Vert^2_{L^2(0,\ell_0)}+\Vert\dot{u}^\epsn_0\Vert^2_{L^2(0,\ell_0)}+\mc A^\epsn({\varphi^\epsn}^{-1}(0))\right).
		\end{align*}
		We thus conclude by means of (H1) and \eqref{energybound}, since ${\varphi^\epsn}^{-1}(0)$ is uniformly bounded with respect to $\epsn$ thanks to Corollary~\ref{ellbd}.
	\end{proof}\noindent
Now we pass to the limit in the energy-dissipation balance \eqref{eb}.
\begin{prop}\label{mu}
	Assume (H1), (K0), $\nu>0$ and let $w$ and $\ell$ be given by \eqref{unifw} and Proposition~\ref{ellconv}, respectively. Then there exists a positive measure $\mu$ on $[0,+\infty)$ for which the following equality holds true for every $t\in[0,+\infty)$:
	\begin{equation*}
		\frac 12\frac{w(t)^2}{\ell^+(t)}+\int_{\ell_0}^{\ell^+(t)}\!\!\!\!\!\!\!\!\!\!\kappa(\sigma)\d\sigma-\int_{0}^{t}\!\!\!\dot{w}(\tau)\frac{w(\tau)}{\ell(\tau)}\d\tau+\mu([0,t])=\liminf\limits_{n\to+\infty}\left(\!\frac 12\!\int_{0}^{\ell_0}\!\!\!\!\!\epsn^2 u_1^\epsn(\sigma)^2\d\sigma+\frac 12\int_{0}^{\ell_0}\!\!\!\!\!\dot{u}_0^\epsn(\sigma)^2\d\sigma\!\right),
	\end{equation*}
	where $\epsn$ is the subsequence given by \eqref{unifw} and by Propositions~\ref{ellconv} and \ref{energydecay}.\\
	Moreover for every $0<s\le t$ one has:
	\begin{equation}\label{inequality}
		\frac 12\frac{w(t)^2}{\ell^+(t)}+\int_{\ell_0}^{\ell^+(t)}\!\!\!\!\!\!\!\!\!\!\!\kappa(\sigma)\d\sigma-\int_{0}^{t}\!\!\!\dot{w}(\tau)\frac{w(\tau)}{\ell(\tau)}\d\tau+\mu([s,t])=\frac 12\frac{w(s)^2}{\ell^-(s)}+\int_{\ell_0}^{\ell^-(s)}\!\!\!\!\!\!\!\!\!\!\!\kappa(\sigma)\d\sigma-\int_{0}^{s}\!\!\!\dot{w}(\tau)\frac{w(\tau)}{\ell(\tau)}\d\tau.
	\end{equation}
\end{prop}
\begin{proof}
	By classical properties of BV functions in one variable (see for instance \cite{AmFuPa}, Theorem~3.28) it is enough to prove that the function $f\colon(-\delta,+\infty)\to \erre$ defined as:
	\begin{equation*}
		f(t):=\begin{cases}\displaystyle
		\liminf\limits_{n\to+\infty}\left(\frac 12\int_{0}^{\ell_0}\!\!\!\!\!\epsn^2 u_1^\epsn(\sigma)^2\d\sigma+\frac 12\int_{0}^{\ell_0}\!\!\!\!\!\dot{u}_0^\epsn(\sigma)^2\d\sigma\right),&\text{if }t\in(-\delta,0],\\
		\displaystyle\frac 12\frac{w(t)^2}{\ell(t)}+\int_{\ell_0}^{\ell(t)}\!\!\!\!\!\!\!\!\!\kappa(\sigma)\d\sigma-\int_{0}^{t}\!\!\!\dot{w}(\tau)\frac{w(\tau)}{\ell(\tau)}\d\tau,&\text{if }t\in(0,+\infty),
		\end{cases}
	\end{equation*}
	belongs to the Lebesgue class of a nonincreasing function. Indeed in that case $\mu:=-Df$ does the job.\par 
	We actually prove that the right limit $f^+$ is nonincreasing. We fix $s,t\in(-\delta,+\infty)$ such that $s<t$ and we consider all the possible cases.\par 
	If $s\ge 0$ we pick two sequences $\{s_j\}_{j\in\enne}$, $\{t_j\}_{j\in\enne}$ such that for every $j\in\enne$ one has $s<s_j<t<t_j$, $s_j$ and $t_j$ do not belong to the jump set of $\ell$, and $s_j\searrow s$, $t_j\searrow t$ as $j\to +\infty$. By the energy-dissipation balance \eqref{eb} we hence get:
	\begin{equation*}
	\mc E^\epsn(t_j)+\int_{\ell_0}^{\ellepsn(t_j)}\!\!\!\!\!\!\!\!\!\!\!\!\kappa(\sigma)\d\sigma+\int_{0}^{t_j}\wepsnd(\tau)\uepsnx(\tau,0)\d\tau\le\mc E^\epsn(s_j)+\int_{\ell_0}^{\ellepsn(s_j)}\!\!\!\!\!\!\!\!\!\!\!\!\kappa(\sigma)\d\sigma+\int_{0}^{s_j}\wepsnd(\tau)\uepsnx(\tau,0)\d\tau.
	\end{equation*}
	Passing to the limit as $n\to+\infty$, by Theorem~\ref{convu} and by exploiting Corollary~\ref{bordo} together with \eqref{unifw} we deduce:
	\begin{equation*}
	\frac 12 \frac{w(t_j)^2}{\ell(t_j)}+\int_{\ell_0}^{\ell(t_j)}\!\!\!\!\!\!\!\!\kappa(\sigma)\d\sigma-\int_{0}^{t_j}\dot{w}(\tau)\frac{w(\tau)}{\ell(\tau)}\d\tau\le\frac 12 \frac{w(s_j)^2}{\ell(s_j)}+\int_{\ell_0}^{\ell(s_j)}\!\!\!\!\!\!\!\!\kappa(\sigma)\d\sigma-\int_{0}^{s_j}\dot{w}(\tau)\frac{w(\tau)}{\ell(\tau)}\d\tau.
	\end{equation*}
	Passing now to the limit as $j\to+\infty$ we get $f^+(t)\le f^+(s)$.\par 
	If $s\in(-\delta,0)$ and $t\ge0$ we consider a sequence $\{t_j\}_{j\in\enne}$ as before and by means of the energy-dissipation balance we infer:
	\begin{equation*}
	\mc E^\epsn(t_j)+\int_{\ell_0}^{\ellepsn(t_j)}\!\!\!\!\!\!\!\!\!\!\!\!\kappa(\sigma)\d\sigma+\int_{0}^{t_j}\wepsnd(\tau)\uepsnx(\tau,0)\d\tau\le\mc E^\epsn(0)=\frac 12\int_{0}^{\ell_0}\!\!\!\!\!\epsn^2 u_1^\epsn(\sigma)^2\d\sigma+\frac 12\int_{0}^{\ell_0}\!\!\!\!\!\dot{u}_0^\epsn(\sigma)^2\d\sigma.
	\end{equation*}
	Passing to the limit as $n\to+\infty$ and then $j\to +\infty$ we hence deduce that also in this case $f^+(t)\le f^+(s)$.\par 
	If finally both $s$ and $t$ belong to $(-\delta,0)$, then trivially $f^+(t)=f^+(s)$ and so we conclude.
\end{proof}
The measure $\mu$ introduced in the previous Theorem somehow represents the amount of energy dissipated by viscosity which still is present in the limit. Indeed it can be seen as a weak$^*$-limit of $\mc A^\eps$ as $\eps\to 0^+$. The rise of such a limit measure occurs also in \cite{Roub} in a model of contact between two visco-elastic bodies. Of course, to obtain the desired quasistatic energy-dissipation balance (eb) we need to prove that $\mu\equiv 0$, namely that $\mc A^\eps$ vanishes as $\eps\to 0^+$. To this aim, we first prove that the limit debonding front $\ell$ is a continuous function; this is, however, a crucial step for getting (eb) from \eqref{inequality}. As in Section \ref{sec2}, to reach the result we need to strenghten the assumptions on the toughness $\kappa$.
\begin{cor}\label{contelle}
	Assume (H1), (K0), (K2) and let $\nu$ be positive. Then the nondecreasing function $\ell$ given by Proposition~\ref{ellconv} is continuous in $(0,+\infty)$.
\end{cor}
\begin{proof}
	The result follows arguing as in the proof of Lemma~\ref{cont} by means of \eqref{stablimit} and \eqref{inequality}; see also Remark~\ref{alsoinequality}.
\end{proof}
\begin{prop}\label{ebquas}
	Assume (H1), (K0), (K3) and let $\nu$ be positive. Then the following energy-dissipation balance holds true for the nondecreasing function $\ell$ obtained in Proposition~\ref{ellconv}:
	\begin{equation}\label{equality}
		\frac 12\frac{w(t)^2}{\ell(t)}+\int_{\ell^+(0)}^{\ell(t)}\kappa(\sigma)\d\sigma-\int_{0}^{t}\dot{w}(\tau)\frac{w(\tau)}{\ell(\tau)}\d\tau=\frac12\frac{w(0)^2}{\ell^+(0)},\quad\text{for every }t\in(0,+\infty),
	\end{equation}
	where $w$ is given by \eqref{unifw}.
\end{prop}
\begin{proof}
	By Corollary~\ref{contelle} we know $\ell$ is continuous on $(0,+\infty)$, by \eqref{stablimplus} we deduce $\ell$ satisfies stability condition (s2) in $(0,+\infty)$, while by \eqref{inequality} the function
	\begin{equation*}
		t\mapsto	\frac 12\frac{w(t)^2}{\ell(t)}+\int_{\ell_0}^{\ell(t)}\kappa(\sigma)\d\sigma-\int_{0}^{t}\dot{w}(\tau)\frac{w(\tau)}{\ell(\tau)}\d\tau,
	\end{equation*}
	is nonincreasing in $(0,+\infty)$. Thus, by Proposition~\ref{explicitlambda} and Remark~\ref{alsoinequality2} we deduce that $\ell$ has the form \eqref{solution}, with $\ell^+(0)$ in place of $\ell_0$. By (K3) and by means of Theorem~\ref{exuniqquas} we hence conclude. Indeed we point out that, under our assumptions, condition (KW) is automatically satisfied: by (K0) and (K2) we deduce $\lim\limits_{x\to+\infty}\phi_\kappa(x)=+\infty$ and by \eqref{stablimplus} we have $\phi_\kappa(\ell^+(0))\ge\frac 12 w(0)^2$.
\end{proof}
Previous Proposition shows that the measure $\mu$ introduced in Theorem~\ref{mu} is concentrated on the singleton $\{0\}$. This means that friction dissipates all the initial energy at the initial time $t=0$.
Up to now we have thus proved that, under suitable assumptions, the limit pair $(u,\ell)$ is a quasistatic evolution starting from the point $\ell^+(0)$. The aim of the next subsection will be characterise the value $\ell^+(0)$.
\subsection{The initial jump}
In this subection we show that the (possible) initial jump of the limit debonding front $\ell$ is characterised by the equality $\ell^+(0)=\lim\limits_{t\to+\infty}\tilde{\ell}(t)$, where $\tilde{\ell}$ is the debonding front related to the unrescaled dynamic coupled problem:
	\begin{equation}
\label{problemtilde}
\begin{cases}
\tilde{u}_{tt}(t,x)-\tilde{u}_{xx}(t,x)+\nu \tilde{u}_t(t,x)=0, \quad& t > 0 \,,\, 0<x<\tilde{\ell}(t),  \\
\tilde{u}(t,0)=w(0), &t>0, \\
\tilde{u}(t,\tilde{\ell}(t))=0,& t>0,\\
\tilde{u}(0,x)=u_0(x),\quad&0<x<\ell_0,\\
\tilde{u}_t(0,x)=0,&0<x<\ell_0,
\end{cases}
\end{equation}
\begin{equation}\label{Griffithtilde}
\begin{cases}
0\le\dot{\tilde{\ell}}(t)<1,\\
{G}_{\dot{\tilde{\ell}}(t)}(t)\le\kappa(\tilde{\ell}(t)),\\
\left[{G}_{\dot{\tilde{\ell}}(t)}(t)-\kappa(\tilde{\ell}(t))\right]\dot{\tilde{\ell}}(t)=0,
\end{cases}\quad\quad\text{ for a.e. }t\in(0,+\infty).
\end{equation}

Here we are assuming that $u_0\in H^1(0,\ell_0)$ satisfies $u_0(0)=w(0)$ and $u_0(\ell_0)=0$. Moreover, as before, we consider $\nu>0$ and a positive toughness $\kappa$ which belongs to $\widetilde{C}^{0,1}([\ell_0,+\infty))$. We also need to introduce stronger conditions than (H1):
\begin{itemize}
	\item[(H2)] the family $\{\weps\}_{\eps>0}$ is bounded in $\widetilde H^1(0,+\infty)$, $u_0^\eps\to u_0$ strongly in $H^1(0,\ell_0)$, $\eps u_1^\eps\to0$ strongly in $L^2(0,\ell_0)$ as $\eps\to 0^+$.
	\item[(H3)] $\weps\rightharpoonup w$ weakly in $\widetilde H^1(0,+\infty)$, $u_0^\eps\to u_0$ strongly in $H^1(0,\ell_0)$, $\eps u_1^\eps\to0$ strongly in $L^2(0,\ell_0)$ as $\eps\to 0^+$.
\end{itemize}
\begin{rmk}
	Assuming (H3), by the compact embedding of $H^1(0,T)$ in $C^0([0,T])$ we deduce that for every $T>0$ we have $\weps\to w$ uniformly in $[0,T]$ as $\eps\to 0^+$.
\end{rmk}
\begin{rmk}
	As explained in Section \ref{sec1} the pair $(\tilde{u},\tilde{\ell})$ solution of \eqref{problemtilde}\&\eqref{Griffithtilde} fulfills the energy-dissipation balance:
	\begin{equation}\label{ebtilde}
		\mc E(t)+\mc A(t)+\int_{\ell_0}^{\tilde{\ell}(t)}\kappa(\sigma)\d\sigma=\frac 12\int_{0}^{\ell_0}\dot{u}_0(\sigma)^2\d\sigma,\quad\text{for every }t\in[0,+\infty),
	\end{equation}
	where $\mc E$ and $\mc A$ are as in \eqref{kinpot} and \eqref{frict} with $\eps=1$ and $\tilde{u}$, $\tilde{\ell}$ in place of $\ueps$ and $\elleps$.
\end{rmk}
We want to notice that, assuming (H2) and considering the subsequence $\epsn$ given by Remark~\ref{unifwrmk}, one can apply Theorem~\ref{contdependence} deducing that actually the pair $(\tilde{u},\tilde{\ell})$ is the limit as $n\to +\infty$ (in the sense of Theorem~\ref{contdependence}) of $(u_\epsn,\ell_\epsn)$, where this last pair is the dynamic evolution related to the unrescaled problem \eqref{problem0} (replacing $w$, $u_0$, $u_1$ by $w^\epsn$, $u_0^\epsn$, $u_1^\epsn$) coupled with dynamic Griffith's criterion.\par 
We denote by $\ell_1$ the limit of $\tilde{\ell}(t)$ when $t$ goes to $+\infty$. Before studying the relationship between $\ell_1$ and $\ell^+(0)$ we perform an asymptotic analysis of the pair $(\tilde{u},\tilde{\ell})$ as $t\to +\infty$.

\begin{lemma}\label{delta}
	Assume (K0). Then for every $\delta>0$ there exists a time $T_\delta>0$ and a measurable set $N_\delta\subseteq(T_\delta,+\infty)$ such that $|N_\delta|\le\delta$ and $\dot{\tilde{\ell}}(t)\le\delta$ for every $t\in(T_\delta,+\infty)\diff N_\delta$.
\end{lemma}
\begin{proof}
	First of all we notice that by (K0) we deduce from the energy-dissipation balance \eqref{ebtilde} that $\ell_1$ is finite. Then we fix $\delta>0$ and we consider $T_\delta>0$ in such a way that $\ell_1-\tilde{\ell}(T_\delta)\le\delta^2$. Introducing the sets:
	\begin{align*}
		&ND_\delta:=\{t>T_\delta\mid \tilde{\ell} \text{ is not differentiable at }t\},\\
		& M_\delta:=\{t>T_\delta\mid \tilde{\ell}\text{ is differentiable at }t\text{ and }\dot{\tilde{\ell}}(t)>\delta \},
	\end{align*}
	we then define $N_\delta:=ND_\delta\cup M_\delta$. By construction $\dot{\tilde{\ell}}(t)\le\delta$ for every $t\in(T_\delta,+\infty)\diff N_\delta$, while by means of Čebyšëv inequality we deduce:
	\begin{align*}
		|N_\delta|=|M_\delta|\le\frac 1\delta \int_{T_\delta}^{+\infty}\dot{\tilde{\ell}}(\tau)\d\tau=\frac{\ell_1-\tilde{\ell}(T_\delta)}{\delta}\le\delta,
	\end{align*}
	and we conclude.
\end{proof}
All the next Propositions trace what we have done in the previous Sections to deal with the analysis of the limit of the pair $(\ueps,\elleps)$ when $\eps\to 0^+$. For this reason the proofs are only sketched.
\begin{prop}\label{limit}
	Assume (K0). Then one has $\displaystyle\lim\limits_{t\to +\infty}{\mc E}(t)=\frac 12\frac{w(0)^2}{\ell_1}$.
\end{prop}
\begin{proof}
	As in Section \ref{sec3} we introduce the modified energy:
	\begin{equation*}
		\widetilde{\mc E}(t):=\frac 12 \int_{0}^{\tilde{\ell}(t)}\tilde{u}_t(t,\sigma)^2\d\sigma+\frac 12\int_{0}^{\tilde\ell(t)}\big(\tilde{u}_x(t,\sigma)-\tilde{r}_x(t,\sigma)\big)^2\d\sigma,\quad \text{for }t\in[0,+\infty),
		\end{equation*}
		where  
		\begin{equation*}
		\tilde{r}(t,x):=w(0)\left(1-\frac{x}{\tilde\ell(t)}\right)\chi_{[0,\tilde\ell(t)]}(x),\quad \text{for }(t,x)\in[0,+\infty)\times[0,+\infty).
		\end{equation*}
		Repeating the proof of Theorem~\ref{mainestimate} we deduce that the following estimate holds true:
		\begin{equation}\label{estimatetilde}
			\widetilde{\mc E}(t)\le 4 \widetilde{\mc E}(0)e^{-mt}+Ce^{-mt}\int_{0}^{t}\dot{\tilde{\ell}}(\tau)e^{m\tau}\d\tau,\quad\text{for every }t\in[0,+\infty),
		\end{equation}
		where $m$ is a suitable positive value and $C$ is a positive constant independent of $t$. By means of Lemma~\ref{delta} now we show that the second term in \eqref{estimatetilde} goes to $0$ when $t\to +\infty$. Indeed let us fix $\delta>0$ and consider $T_\delta$, $N_\delta$ as in Lemma~\ref{delta}; then for every $t\ge T_\delta$ we can estimate:
		\begin{align*}
			e^{-mt}\int_{0}^{t}\dot{\tilde{\ell}}(\tau)e^{m\tau}\d\tau&=e^{-mt}\left(\int_{0}^{T_\delta}\dot{\tilde{\ell}}(\tau)e^{m\tau}\d\tau+\int\limits_{(T_\delta,t)\cap N_\delta}\dot{\tilde{\ell}}(\tau)e^{m\tau}\d\tau+\int\limits_{(T_\delta,t)\setminus N_\delta}\dot{\tilde{\ell}}(\tau)e^{m\tau}\d\tau\right)\\
			&\le e^{-mt}\left(\int_{0}^{T_\delta}\dot{\tilde{\ell}}(\tau)e^{m\tau}\d\tau+e^{mt}|N_\delta|+\delta\int_{T_\delta}^{t}e^{m\tau}\d\tau\right)\\
			&\le e^{-mt}\int_{0}^{T_\delta}\dot{\tilde{\ell}}(\tau)e^{m\tau}\d\tau+\delta\left(1+\frac 1m\right).
		\end{align*}
		Letting first $t\to+\infty$ and then $\delta\to0^+$ we hence deduce that $\displaystyle\lim\limits_{t\to +\infty}e^{-mt}\int_{0}^{t}\dot{\tilde{\ell}}(\tau)e^{m\tau}\d\tau=0$ and so we get $\lim\limits_{t\to +\infty}\widetilde{\mc E}(t)=0$. Now we conclude since like in \eqref{squares} we have:
		\begin{equation*}
			\mc E(t)=\widetilde{\mc E}(t)+\frac 12 \frac{w(0)^2}{\tilde\ell(t)},\quad\text{for every }t\in[0,+\infty).
		\end{equation*}
\end{proof}
\begin{lemma}\label{boundarytilde}
	Assume (K0). Then the following limit holds true:
	\begin{equation*}
		\lim\limits_{t\to +\infty}\frac 1t\int_{0}^{t}\left(\tilde{u}_x(\sigma,0)+\frac{w(0)}{\tilde\ell(\tau)}\right)^2\d\tau=0.
	\end{equation*}
\end{lemma}
\begin{proof}
	The proof is analogous to the one of Corollary~\ref{bordo}. By using \eqref{intparts2} with the obvious changes, for every $t>0$ we obtain the estimate:
	\begin{equation*}
		\int_{0}^{t}\left(\tilde{u}_x(\sigma,0)+\frac{w(0)}{\tilde\ell(\tau)}\right)^2\d\tau\le C\left(\int_{0}^{t}\widetilde{\mc E}(\tau)\d\tau+\mc E(t)+\tilde\ell(t)\right)\le C\left(\int_{0}^{t}\widetilde{\mc E}(\tau)\d\tau+\mc E(0)+\ell_1\right).
	\end{equation*}
	From this we conclude by applying de l'Hôpital's rule since in Proposition~\ref{limit} we proved that $\lim\limits_{t\to +\infty}\widetilde{\mc E}(t)=0$.
\end{proof}
\begin{prop}\label{lessder}
	Assume (K0). Then $\ell_1$ satisfies the stability condition at time $t=0$, namely:
	\begin{equation*}
		\frac 12\frac{w(0)^2}{\ell_1^2}\le\kappa(\ell_1).
	\end{equation*} 
\end{prop}
\begin{proof}
	The idea is to pass to the limit as $t\to +\infty$ in the stability condition in Griffith's criterion \eqref{Griffithtilde}, as we did in Proposition~\ref{stabprop}. Since here we want to compute a limit when $t$ grows to $+\infty$, as in Lemma~\ref{boundarytilde} we need to average the stability condition, getting:
	\begin{equation}\label{average}
		\frac 1t\int_{0}^{t}\kappa(\tilde{\ell}(\sigma))\d\sigma\ge\frac 1t\int_{0}^{t}{G}_{\dot{\tilde{\ell}}(\sigma)}(\sigma)\d\sigma,\quad\quad\text{ for every }t\in(0,+\infty).
	\end{equation}
	By de l'Hôpital's rule the left-hand side in \eqref{average} converges to $\kappa(\ell_1)$ as $t\to +\infty$, while to deal with the right-hand side we argue as in the proof of Proposition~\ref{stabprop}. For the sake of simplicity we introduce the time $t^*>0$ which satisfies $t^*=\tilde{\ell}(t^*)$, so that for every $t\ge t^*$ we can write:
	\begin{equation}\label{booh}
	\begin{aligned}
		\frac 1t\int_{0}^{t}{G}_{\dot{\tilde{\ell}}(\sigma)}(\sigma)\d\sigma&\ge\frac 1t\int_{t^*}^{t}{G}_{\dot{\tilde{\ell}}(\sigma)}(\sigma)\d\sigma\\
		&=\frac 1t\int_{0}^{\tilde\varphi(t)}\frac{1}{\dot{\tilde{\psi}}(\tilde{\varphi}^{-1}(\sigma))}\frac 12 \left(\tilde{u}_x(\sigma,0)+\nu\int_{\sigma}^{\tilde{\varphi}^{-1}(\sigma)}\tilde{u}_t(\tau,\tau-\sigma)\d\tau\right)^2\d\sigma,
	\end{aligned}
	\end{equation}
	where we used the explicit formula for ${G}_{\dot{\tilde{\ell}}(\sigma)}(\sigma)$ given by \eqref{dynenrate} and \eqref{F}, with the obvious changes. By means of Lemma~\ref{boundarytilde} and since $\lim\limits_{t\to +\infty}\frac{\tilde{\varphi}(t)}{t}=\lim\limits_{t\to +\infty}\frac{t-\tilde\ell(t)}{t}=1$ it is easy to infer:
	\begin{equation}\label{uno}
		\lim\limits_{t\to +\infty}\frac 1t\int_{0}^{\tilde\varphi(t)}\frac{1}{\dot{\tilde{\psi}}(\tilde{\varphi}^{-1}(\sigma))}\frac 12\tilde{u}_x(\sigma,0)^2\d\sigma=\frac 12 \frac{w(0)^2}{\ell_1^2}.
	\end{equation}
	Moreover, by using estimate \eqref{estfriction} in the proof of Proposition~\ref{stabprop} and recalling that the dissipated energy $\mc A$ is bounded by \eqref{ebtilde}, we deduce:
	\begin{equation}\label{due}
		\lim\limits_{t\to +\infty}\frac 1t\int_{0}^{\tilde\varphi(t)}\frac{1}{\dot{\tilde{\psi}}(\tilde{\varphi}^{-1}(\sigma))}\left(\int_{\sigma}^{\tilde{\varphi}^{-1}(\sigma)}\tilde{u}_t(\tau,\tau-\sigma)\d\tau\right)^2\d\sigma=0.
	\end{equation}
	From \eqref{uno} and \eqref{due} we can pass to the limit in \eqref{booh} deducing that:
	\begin{equation*}
		\liminf\limits_{t\to +\infty}\frac 1t\int_{0}^{t}{G}_{\dot{\tilde{\ell}}(\sigma)}(\sigma)\d\sigma\ge \frac 12 \frac{w(0)^2}{\ell_1^2},
	\end{equation*}
	and so we conclude.
\end{proof}\noindent
We are now in a position to compare the value of $\ell^+(0)$ with $\ell_1$.
\begin{lemma}\label{less}
	Assume (H2) and (K0). Then $\ell_1\le\ell^+(0)$.
\end{lemma}
\begin{proof}
	We fix $t>0$ and we consider the subsequence $\epsn\searrow 0$ given by Remark~\ref{unifwrmk} and Proposition~\ref{ellconv}. Then one has:
	\begin{equation*}
	\ell(t)=\lim\limits_{n\to +\infty}\ellepsn(t)=\lim\limits_{n\to +\infty}\ell_\epsn\left(\frac{t}{\epsn}\right).
	\end{equation*}
	Now we fix $T>0$ and by monotonicity we deduce $\ell_\epsn\left(\frac{t}{\epsn}\right)\ge\ell_\epsn\left(T\right)$ for $n$ large enough. Thus, by means of Theorem~\ref{contdependence}, we get:
	\begin{equation*}
	\lim\limits_{n\to +\infty}\ell_\epsn\left(\frac{t}{\epsn}\right)\ge \lim\limits_{n\to +\infty}\ell_\epsn\left(T\right)=\tilde{\ell}(T).
	\end{equation*}
	Hence $\ell(t)\ge\tilde{\ell}(T)$ and by the arbitrariness of $t>0$ and $T>0$ we conclude.
\end{proof}
\begin{prop}\label{energyless}
	Assume (H2), (K0) and (K3). Then the following inequality holds true:
	\begin{equation*}
		\frac 12\frac{w(0)^2}{\ell^+(0)}+\int_{\ell_0}^{\ell^+(0)}\kappa(\sigma)\d\sigma\le\frac 12\frac{w(0)^2}{\ell_1}+\int_{\ell_0}^{\ell_1}\kappa(\sigma)\d\sigma.
	\end{equation*}
\end{prop}
\begin{proof}
	By Proposition~\ref{mu}, Corollary~\ref{contelle} and the energy-dissipation balance \eqref{eb} we know that for every $t>0$ it holds:
	\begin{equation*}
		\lim\limits_{n\to +\infty}\mc A^\epsn(t)=\mu([0,t])=\frac 12 \int_{0}^{\ell_0}\dot{u}_0(\sigma)\d\sigma-\frac 12\frac{w(t)^2}{\ell(t)}-\int_{\ell_0}^{\ell(t)}\!\!\!\!\!\!\!\!\kappa(\sigma)\d\sigma+\int_{0}^{t}\!\!\!\dot{w}(\tau)\frac{w(\tau)}{\ell(\tau)}\d\tau,
	\end{equation*}
	where $\epsn$ is the subsequence given by \eqref{unifw} and by Propositions~\ref{ellconv} and \ref{energydecay}. By means of \eqref{equality} we hence deduce:
	\begin{equation}\label{one}
		\lim\limits_{n\to +\infty}\mc A^\epsn(t)=\frac 12 \int_{0}^{\ell_0}\dot{u}_0(\sigma)\d\sigma-	\frac 12\frac{w(0)^2}{\ell^+(0)}-\int_{\ell_0}^{\ell^+(0)}\kappa(\sigma)\d\sigma.
	\end{equation}
	By a simple change of variable we now notice that:
	\begin{equation*}
		\mc A^\epsn(t)=\nu \int_{0}^{t/\epsn}\int_{0}^{\ell_\epsn(\tau)}(u_\epsn)_t(\tau,\sigma)^2\d\sigma\d\tau\ge\nu\int_{0}^{t}\int_{0}^{\ell_\epsn(\tau)}(u_\epsn)_t(\tau,\sigma)^2\d\sigma\d\tau,
	\end{equation*}
	and so, by Theorem~\ref{contdependence}, we get:
	\begin{equation}\label{two}
		\lim\limits_{n\to +\infty}\mc A^\epsn(t)\ge\nu\int_{0}^{t}\int_{0}^{\tilde{\ell}(\tau)}\tilde{u}_t(\tau,\sigma)^2\d\sigma\d\tau.
	\end{equation}
	Putting together \eqref{one} and \eqref{two} we finally deduce:
	\begin{equation*}
		\frac 12 \int_{0}^{\ell_0}\dot{u}_0(\sigma)\d\sigma-	\frac 12\frac{w(0)^2}{\ell^+(0)}-\int_{\ell_0}^{\ell^+(0)}\kappa(\sigma)\d\sigma\ge \lim\limits_{t\to+\infty}\nu\int_{0}^{t}\int_{0}^{\tilde{\ell}(\tau)}\tilde{u}_t(\tau,\sigma)^2\d\sigma\d\tau=\lim\limits_{t\to +\infty}{\mc A}(t).
	\end{equation*}
	To conclude it is enough to recall that by energy-dissipation balance \eqref{ebtilde} we have:
	\begin{equation*}
		{\mc A}(t)=\frac 12 \int_{0}^{\ell_0}\dot{u}_0(\sigma)^2\d\sigma-{\mc E}(t)-\int_{\ell_0}^{\tilde{\ell}(t)}\kappa(\sigma)\d\sigma,\quad\text{ for every }t\in[0,+\infty),
	\end{equation*}
	and so by Proposition~\ref{limit} we obtain:
	\begin{equation*}
	\lim\limits_{t\to +\infty}{\mc A}(t)=\frac 12 \int_{0}^{\ell_0}\dot{u}_0(\sigma)^2\d\sigma-\frac 12 \frac{w(0)^2}{\ell_1}-\int_{\ell_0}^{\ell_1}\kappa(\sigma)\d\sigma.
	\end{equation*}
\end{proof}
\begin{cor}
	Assume (H2), (K0) and (K3). Then $\ell_1=\ell^+(0)$.
\end{cor}
\begin{proof}
	By Lemma~\ref{less} we already know that $\ell_1\le\ell^+(0)$. As in Proposition~\ref{equiv} we introduce the energy:
	\begin{equation*}
		E_0(x):=\frac 12 \frac{w(0)^2}{x}+\int_{\ell_0}^{x}\kappa(\sigma)\d\sigma,\quad\quad\text{for }x\in[\ell_0,+\infty).
	\end{equation*}
	By Proposition~\ref{energyless} we get $E_0(\ell^+(0))\le E_0(\ell_1)$, while by Proposition~\ref{lessder} and (K3) we deduce that $\dot{E}_0(x)>0$ for every $x>\ell_1$, namely $E_0$ is strictly increasing in $(\ell_1,+\infty)$. Thus we finally obtain $\ell_1=\ell^+(0)$.
\end{proof}\noindent
Putting together all the results obtained up to now we can finally deduce our main Theorem:
\begin{thm}\label{finalthm}
	Fix $\nu> 0$, $\ell_0>0$ and assume the functions $\weps$, $u_0^\eps$ and $u_1^\eps$ satisfy \eqref{bdryregularity} and \eqref{compatibility0} for every $\eps>0$. Let the positive toughness $\kappa$ belong to $\widetilde{C}^{0,1}([\ell_0,+\infty))$ and assume (H2), (K0) and (K3). Let $(\ueps,\elleps)$ be the pair of dynamic evolutions given by Theorem~\ref{exuniq}. Let $\epsn$ and $w$ be the subsequence and the function given by Remark~\ref{unifwrmk} and let $\ell_1$ be defined as $\ell_1:=\lim\limits_{t\to +\infty}\tilde{\ell}(t)$, with $(\tilde{u},\tilde{\ell})$ solution of \eqref{problemtilde}\&\eqref{Griffithtilde}. Then for every $t\in(0,+\infty)$ one has:
	\begin{itemize}
		\item[(a)] $\lim\limits_{n\to+\infty}\ellepsn(t)=\ell(t)$,\\
		\item[(b)] $\epsn\uepsnt(t,\cdot)\to 0$ strongly in $L^2(0,+\infty)$ as $n\to +\infty$,\\
		\item[(c)] $\uepsn(t,\cdot)\to u(t,\cdot)$ strongly in $H^1(0,+\infty)$ as $n\to +\infty$,
	\end{itemize}
 where $(u,\ell)$ is the quasistatic evolution given by Theorem~\ref{exuniqquas} starting from $\ell_1$ and with external loading $w$.\par 
Moreover, if we assume (H3), then we do not need to pass to a subsequence and  the whole sequence $(\ueps,\elleps)$ converges to $(u,\ell)$ in the sense of \textnormal{(a), (b), (c)} for every $t\in(0,+\infty)$ as $\eps\to 0^+$.
\end{thm}
\begin{rmk}
	Under the same assumptions of the above Theorem the convergence of the debonding fronts can be slightly improved by classical arguments. Indeed, since $\elleps$ are nondecreasing continuous functions and since the pointwise limit $\ell$ is continuous in $(0,+\infty)$, we can infer that the convergence stated in (a) is actually uniform on compact sets contained in $(0,+\infty)$.\par
	We want also to recall that for every $T>0$ the convergences in (b) and (c) holds true respectively in $L^2(0,T;L^2(0,+\infty))$ and $L^2(0,T;H^1(0,+\infty))$ too, as we proved in Proposition~\ref{energydecay} under weaker assumptions.
\end{rmk}
\begin{rmk}
	We want to notice that Theorem~\ref{exuniqquas} ensures the limit $\ell$ is an absolutely continuous function, so one could guess that the convergence of the debonding fronts $\elleps$ even occurs in $W^{1,1}_{\textnormal {loc}}(0,+\infty)$, but unfortunately we were not able to prove it. Of course this last conjecture could be true only under the assumptions of Theorem~\ref{finalthm}, otherwise neither the continuity of the function $\ell$ is expected. Our idea to attack the problem was getting good a priori bounds on $\ellepsd$ via the explicit formula \eqref{explicitder}, but we found the task hard due to the high nonlinearity of the formula. Thus better ideas or better strategies are needed to validate or to disprove our conjecture.
\end{rmk}
\section{Conclusions}
In this paper we have proved that dynamic evolutions of a damped debonding model are a good approximation of the quasistatic one when initial velocity and speed of the external loading are very slow with respect to internal vibrations. In light of \cite{LazNar}, in which the failure of this approximation in the undamped case (even with constant toughness) is shown, it is clear that the presence of viscosity, or more generally the presence of some kind of friction, is crucial to get this kind of result. As previously said, the importance of viscosity was already observed in finite dimension and in some damage models.\par 
Although in our work we have been able to cover cases of quite general toughness $\kappa$, we however needed to require some assumptions on it to develop all the arguments. First of all we have always assumed continuity of $\kappa$ and furthermore conditions (K1), (K2) or (K3) have been used to prevent the case of a glue whose toughness oscillates dramatically. It is worth noticing that we did not make use of them until Corollary~\ref{contelle}, thus our previous analysis is suited to deal with wild oscillating (but still continuous) toughnesses too. Going further in the analysis without that assumptions requires a deep understanding of the measure $\mu$ introduced in Proposition~\ref{mu}. This kind of study has been developed in \cite{ScilSol} in finite dimension, but a generalisation to our infinite dimensional setting seemed hard to us. The idea in \cite{ScilSol} relies on the introduction of a suitable cost function which measures the energy gap of a limit solution after a jump in time, and hence characterises their counterpart of measure $\mu$.\par 
It is easy to imagine that without condition (K1) we lose uniqueness and continuity (in time) of quasistatic evolutions, since in that case local but not global minima of the energy can exist. However a more careful analysis on the quasistatic limit could be useful to select and characterise those quasistatic evolutions coming from dynamic ones, and thus somehow more physical.\par 
A more drastic scenario may even appear in the case of a discontinuous toughness, covered however by Theorem~\ref{exuniq}. The failure of the quasistatic approximation in this framework was observed in \cite{DouMarCha08} and \cite{LBDM12} where the authors considered explicit examples of piecewise constant toughness $\kappa$; they noticed that on discontinuity points of $\kappa$ the limit solution does not fulfill Griffith's criterion, which has to be replaced by a suitable energy balance. This is in line with Proposition~\ref{mu} (which however should be proven without assuming continuity of $\kappa$), where the appearance of the measure $\mu$ in \eqref{inequality} takes into account this feature. A similar phenomenon emerges in \cite{Roub} too. As we said before a more complete comprehension of $\mu$ may thus open new perspectives in the understanding of the topic of quasistatic limit.\par 
Finally we want to mention that different kind of frictions may be considered in the dynamic model, replacing the viscous term $u_t(t,x)$ in the wave equation for instance by ${-}u_{txx}(t,x)$ (Kelvin-Voigt model, see \cite{DautLion}, \cite{Slep}) or by a convolution term of the form $\int_{0}^{+\infty}{-}h(\tau)u_{txx}(t-\tau,x)\d\tau$ (viscoelastic materials, see \cite{Conti}, \cite{DautLion}, \cite{Slep}). To our knowledge an analysis of debonding models under the action of these kind of viscoelastic dampings is still missing in literature.\par
We leave all of these questions and proposals open to further research.

	\bigskip
	
	\noindent\textbf{Acknowledgements.}
	The author wishes to thank Prof. Gianni Dal Maso for many helpful discussions on the topic. The author is a member of the Gruppo Nazionale per l'Analisi Matematica, la Probabilit\`a e le loro Applicazioni (GNAMPA) of the Istituto Nazionale di Alta Matematica (INdAM).

	\bigskip
	
	\bibliographystyle{siam}

	{\small
		\vspace{15pt} (Filippo Riva) SISSA, \textsc{Via Bonomea, 265, 34136, Trieste, Italy}
		\par 
		\textit{e-mail address}: \textsf{firiva@sissa.it}
		\par
	}

\end{document}